\documentclass[final,1p,times,authoryear]{elsarticle}

\usepackage[utf8]{inputenc}
\usepackage[T1]{fontenc}
\usepackage{amsmath}%
\usepackage{amsfonts}%
\usepackage{amsthm}
\usepackage{amssymb}%
\usepackage{txfonts}
\usepackage{graphicx}
\usepackage{caption}
\usepackage{subcaption}
\usepackage{bbm}
\usepackage{stmaryrd}
\usepackage{epstopdf}
\usepackage{here}
\usepackage{marvosym,wasysym}
\usepackage{fancyhdr}
\usepackage{hyperref}
\usepackage{natbib}
\usepackage{multirow}
\usepackage{enumitem}

\usepackage{mathtools}
\DeclarePairedDelimiter{\floor}{\lfloor}{\rfloor}

\theoremstyle{plain}
\newtheorem{theorem}{Theorem}[section]
\newtheorem{corollary}[theorem]{Corollary}
\newtheorem{lemma}[theorem]{Lemma}
\newtheorem{proposition}[theorem]{Proposition}

\newtheorem{condition}[theorem]{Condition}

\theoremstyle{definition}
\newtheorem{definition}[theorem]{Definition}

\newtheorem{algorithm}[theorem]{Algorithm}

\theoremstyle{remark}

\numberwithin{equation}{section}

\newcommand{\N}{\mathbbm{N}}
\newcommand{\R}{\mathbbm{R}}

\newcommand{\Z}{\mathbbm{Z}}
\newcommand{\I}{\mathbbm{1}}

\newcommand{\p}{\mathbbm{P}}

\newcommand{\cA}{\mathcal{A}}
\newcommand{\cB}{\mathcal{B}}

\newcommand{\cD}{\mathcal{D}}

\newcommand{\cG}{\mathcal{G}}
\newcommand{\cF}{\mathcal{F}}
\newcommand{\cH}{\mathcal{H}}
\newcommand{\cI}{\mathcal{I}}
\newcommand{\cL}{\mathcal{L}}

\newcommand{\cN}{\mathcal{N}}

\newcommand{\cO}{\mathcal{O}}
\newcommand{\cV}{\mathcal{V}}

\newcommand{\fS}{\mathfrak{S}}

\newcommand{\E}[1]{\mathbbm{E}\left [ \, #1 \, \right ]}

\renewcommand{\epsilon}{\varepsilon}
\renewcommand{\phi}{\varphi}

\newcommand{\pspace}{(\Omega,\cA,\p)}

\newcommand{\intd}[1]{\,\mathrm{d}#1}

\newcommand{\norm}[1]{\left\lVert #1 \right\rVert}
\newcommand{\scalar}[2]{\left\langle #1,#2 \right\rangle}

\newcommand{\argmin}[1]{\operatorname*{arg\,min}_{#1}\,}
\newcommand{\dzn}[1]{d_{\infty}\left( #1 \right) }
\newcommand{\darg}{\,\cdot\,}
\newcommand{\supp}{\text{supp}\,}
\allowdisplaybreaks 

\journal{the Journal of Statistical Planning and Inference}

\begin{document}
\begin{frontmatter}

\title{Orthogonal Series Estimates on Strong Spatial Mixing Data}
\author{Johannes T. N. Krebs\footnote{This research was supported by the Fraunhofer ITWM, 67663 Kaiserslautern, Germany which is part of the Fraunhofer Gesellschaft zur F{\"o}rderung der angewandten Forschung e.V.}
}
\ead{krebs@mathematik.uni-kl.de}
\address{Department of Mathematics, University of Kaiserslautern, Erwin-Schr\"{o}dinger-Stra{\ss}e, 67653 Kaiserslautern.}

\begin{abstract}
We study a nonparametric regression model for sample data which is defined on an $N$-dimensional lattice structure and which is assumed to be strong spatial mixing: we use design adapted multidimensional Haar wavelets which form an orthonormal system w.r.t. the empirical measure of the sample data. For such orthonormal systems, we consider a nonparametric hard thresholding estimator. We give sufficient criteria for the consistency of this estimator and we derive rates of convergence. The theorems reveal that our estimator is able to adapt to the local smoothness of the underlying regression function and the design distribution. We illustrate our results with simulated examples.
\end{abstract}

\end{frontmatter}

\section{Introduction}\label{Introduction}

In this article we study penalized nonparametric sieve estimators for spatial sample data which features a certain dependence structure: the data is given by the random field $(X,Y)$ which is indexed by a set $S$ of spatial coordinates and which is strong spatial mixing. Here, we take $S = \Z^N$ for some lattice dimension $N \in \N_+$ but our discussion is not limited to that regular case; we could also allow that the random field is only partially observed at some $S\subseteq \Z^N$.\\
The random variables $X(s)$ are $\R^d$-valued and have equal marginal distributions denoted by the probability measure $\mu$ on the Borel-$\sigma$-algebra of $\R^d$, $\cB(\R^d)$. The $Y(s)$ are $\R$-valued, square integrable and satisfy the equation
\begin{align}
		Y(s) = m(X(s)) + \varsigma(X(s))\, \epsilon(s), \text{ for each } s \in S \label{lsqI}
\end{align}
where $m, \varsigma: \R^d \rightarrow \R$ are functions in $L^2(\mu)$. The error terms $\epsilon(s)$ are distributed with mean zero and variance one, i.e., $\epsilon(s) \sim (0,1)$. Furthermore, they are independent of $X$ and have identical marginal distributions but may be dependent among each other such that the strong spatial mixing property remains valid. We emphasize that there is no requirement on the distribution of the error terms, e.g., a Gaussian distribution is not necessary. The same is true for the distribution of the regressors $X(s)$, it is not required that these admit a density with respect to the Lebesgue measure.\\
Thus, we apply the classical heteroscedastic regression model to spatial data under minimal assumptions on the random field $(X,Y)$. An introduction to spatial statistics is given by \cite{cressie1993statistics}. In particular, Markov random fields are studied in the monograph of \cite{kindermann1980markov}.\\  
Nonparametric regression on spatial data has gained importance, in particular, the case where the data is given on a regular lattice structure: \cite{hallin2004local} study a local linear kernel estimator under a strong spatial mixing condition. \cite{li2016nonparametric} considers a nonparametric regression estimator for such lattice data which is constructed with wavelets.\\ 
In this article we consider a nonparametric estimator for lattice data, too, however, we do this with a penalized orthogonal series estimator. \cite{baraud2001adaptive} consider penalized estimators for $\beta$-mixing time series $\{(X_t,Y_t):t\in\N\}$ where the regressors $X_t$ are multidimensional. Orthogonal series estimators have been studied for various data situations: for a real-valued one-dimensional regressor $X$ a popular choice are piecewise polynomials. \cite{comte2004new} study an algorithm for the construction in the case of fixed design regression. \cite{kohler2003nonlinear} gives a generalization to random design regression under the assumption that the error terms are bounded. \cite{akakpo2011inhomogeneous} use piecewise polynomials for conditional density estimation of a $\beta$-mixing time series $\{(X_t,Y_t):t\in\N\}$. In another article \cite{kohler2008nonlinear} considers Haar wavelets to construct an orthogonal series estimator in the case of a multivariate regressor $X$ under the assumption of sub Gaussian error terms and a bounded design distribution of $X$. The ideas and results obtained in the latter can be considered as the starting point for our analysis.\\
Before we give a more thorough introduction to the results of this article, we mention that there exist alternative approaches to construct orthogonal series estimators for a random (univariate) regressor $X$. \cite{kerkyacharian2004regression} consider warped wavelets in the case where the regressor $X$ admits a density on a compact real interval. \cite{kulik2009wavelet} use this concept to study time series with long range dependence errors. \cite{FrankeSachs} construct a soft thresholding regression estimator for univariate i.i.d.\ sample data. They derive rates of convergence for Hölder continuous regression functions in a model where the design variables $X$ are supposed to admit a density which has bounded support. \cite{Girardi1997} show that design adapted Haar wavelets can generate even a multiresolution analysis in the one-dimensional case. \cite{masry2000wavelet} studies $\alpha$-mixing stationary processes and derives rates of convergence for regression functions which belong to a multidimensional Besov space.\\
In this article, we transfer the ideas of \cite{kohler2008nonlinear} to the spatial setting where the sample data is no longer independently distributed but where the dependence vanishes with an increasing lattice distance between the random variables. We relax most restrictions which are usually made in the context of nonparametric regression on dependent data. Most notably, the design distribution (which is the distribution of the $X(s)$) does not need to be known and is not restricted to a bounded domain. Furthermore, the distribution does not need to admit a density w.r.t.\ the Lebesgue measure as it is for instance assumed in \cite{hallin2004local}. \cite{li2016nonparametric} assumes in the spatial wavelet regression model that the $X(s)$ admit a density which is known. We do not do this here. Additionally, we do not require the error terms in the regression model to be bounded or sub Gaussian; we develop our results here for a general class of error terms which satisfies a certain condition on the tail distribution. In addition in order to show that the estimator is consistent in the $L^2$-sense, we do not need a bounded regression function. \\
In this paper we establish general consistency results for our nonparametric regression estimators and we derive rates of convergence. Since our assumptions on the distribution of the regressor $X$ and on the error terms $\epsilon$ are less restrictive than usual, we obtain, however, a sub-optimal rate of convergence, when compared to the results of \cite{stone1982optimal}. We shall discuss this further in the corresponding parts of the article. \\
The remainder of the paper is organized as follows: we give the notation and definitions which we use throughout the article in Section~\ref{Section_NotationAndDefinitions}. In Section~\ref{Section_NonlinRegWavelets} we present the main results: we give a general consistency theorem for our nonparametric estimator and derive a rate of convergence theorem. In Section~\ref{Section_SimulationExamples}, we give numerical applications and make the comparison with i.i.d.\ data. The proofs of our theorems are presented in Section~\ref{Section_Proofs}. \ref{Appendix_ExpInequalities} and \ref{AppendixB} contain certain deferred proofs and further background material which proves to be useful in the broader context of random fields. Furthermore, we provide in a supplemental \cite{krebsOrthogonalSuppl} some technical results concerning our simulation procedure.

\section{Notation and Definitions}\label{Section_NotationAndDefinitions}
We work on a probability space $\pspace$ which is equipped with a generic random field $Z$. $Z$ is $\R^d$-valued and is indexed by $\Z^N$, for both $N, d \in \N_+$. This means $Z = \{ Z(s): s\in \Z^N \}$ and $Z(s): \Omega\rightarrow \R^d$ is Borel-measurable for each $s\in \Z^N$. The random field $Z$ is stationary (or homogeneous) if for each translation $t\in\Z^N$ and for each collection of finite points $s_1,\ldots,s_n$ the joint distribution of $\{ Z(s_1+t),\ldots,Z(s_n+t) \}$ coincides with the joint distribution of $\{ Z(s_1),\ldots,Z(s_n) \}$, i.e.,
$$ 	\cL\left(	 Z(s_1+t),\ldots,Z(s_n+t) 	\right) = \cL\left(	 Z(s_1),\ldots,Z(s_n) 	\right). $$
We denote by $\norm{\,\cdot\,}_{p}$ the Euclidean $p$-norm on $\R^N$ and by $d_{p}$ the corresponding metric for $p\in[1,\infty]$ with the extension $d_p(I,J) \coloneqq \inf\{	d_p(s,t), s\in I, t\in J	\}$ for subsets $I,J$ of $\R^N$. Furthermore, write $s \le t$ for $s,t \in \R^N$ if and only if for each $1\le k \le N$ the single coordinates satisfy $s_k \le t_k$. We denote the indicator function of a set $A$ by $\I\{A\}$ and abbreviate for a subset $I$ of $\Z^N$ by $\cF(I) = \sigma\{ Z(s): s\in I\}$ the $\sigma$-algebra generated by the $Z(s)$, $s \in I$.\\ 
As a measure of spatial dependence we use the  $\alpha$-mixing coefficient. This coefficient is introduced by \cite{rosenblatt1956central}; in the spatial context, it is given for $k\in\N$ as
\begin{align}\label{StrongSpatialMixing}
	\alpha(k) \coloneqq \sup_{ \substack{I, J \subseteq \Z^N,\\ \dzn{I,J}\ge k }} \sup_{ \substack{ A \in \cF(I),\\ B \in \cF(J) } } \left| \p(A\cap B)-\p(A)\p(B)		\right|
\end{align}
The random field $Z$ is strong spatial mixing if $\alpha(k) \rightarrow 0$ for $k \rightarrow \infty$. Furthermore, the random field is strong spatial mixing with exponentially decreasing mixing coefficients if there are $c_0$ and $c_1$ in $\R_+$ such that the $\alpha$-mixing coefficients can be bounded as $\alpha(k) \le c_0 \exp(-c_1 k)$. Note that there exist even stronger measures of dependence, we do not need these here. A comprehensive survey on dependence measures is given in \cite{bradley2005basicMixing}.\\
We denote by $e_N \coloneqq (1,\ldots,1)^T$ the $N$-dimensional vector whose entries are equal to 1. For an $N$-dimensional cube in $\Z^N$ that is spanned by two points $a,b \in \Z^N$, $a\le b$, we write $[a..b] = \{s\in\Z^N: a \le s\le b\}$. For instance, let $n=(n_1,\ldots,n_N )^T\in \N_+^N$ then we define the observation domain $I_n$ as
$$ I_n \coloneqq [e_N .. n] =  \left\{s\in \Z^N: e_N \le s \le n \right\}. $$
The observation domain is indexed by a sequence $\left(n(k): k\in\N_+ \right) \subseteq \N^N$ which is increasing in the sense that $n(k) \le n(k+1)$. This sequence has the additional properties that
\begin{align}\label{Cond_Sequence}
	\inf_{k\in\N_+} \frac{ \min\{ n_i(k): i=1,\ldots,N \}}{ \max\{ n_i(k): i=1,\ldots,N \} } > 0
\end{align}
and
\begin{align}\label{Cond_Sequence2}
	\lim_{k\rightarrow\infty} \max\{ n_i(k): i=1,\ldots,N \} = \infty.
\end{align}
Assumption~\eqref{Cond_Sequence2} on the sequence $n(k)$ is obviously needed for a consistent estimator. Assumption~\eqref{Cond_Sequence} is technical and ensures that certain constants exist in the large deviation inequalities which we derive in \ref{Appendix_ExpInequalities}.\\
In order to estimate the coefficients of the hard thresholding estimator, we need a set of functions which is orthonormal w.r.t.\ to the empirical measure $\mu_n = |I_n|^{-1} \sum_{s\in I_n} \delta_{X(s)}$ for a sample of predictor variables $\{X(s): s\in I_n\}$, $I_n \subseteq \Z^N$. Therefore define the empirical scalar product by
\begin{align*}
		\scalar{f}{g}_n \coloneqq \frac{1}{|I_n|} \sum_{s\in I_n} f(X(s) ) g(X(s) ) = \int_{\R^d} f\,g\, \intd{\mu}_n.
\end{align*}
The support of a function $f: \R^d\rightarrow\R$ is $\supp{f} \coloneqq \{ x \in \R^d: f(x) \neq 0 \}$.\\

We summarize our regularity assumptions (on a generic random field $Z$) in a condition which we shall use throughout the rest of this article

\begin{condition}\label{regCond0}
Let $d\in\N_+$. $Z = \left\{Z(s): s\in \Z^N \right\}$ is an $\R^d$-valued random field such that
\begin{enumerate}[label=\textnormal{(\arabic*)}]
\item each $Z(s)$ has the same distribution $\mu$ on $\R^d$. \label{Cond_Distribution}
\item $Z$ is strong spatial mixing with exponentially decreasing mixing coefficients; denote the generic constants which give the bound on the mixing coefficients from \eqref{StrongSpatialMixing} by $c_0$ and $c_1$. \label{Cond_Mixing}
\end{enumerate}
\end{condition}

The assumption on the exponential decay of the mixing rates is quite common, compare for instance \cite{li2016nonparametric}. It is shown in the case of time series by \cite{davydov1973mixing} or \cite{Withers1981} that exponentially decreasing $\alpha$-mixing coefficients are guaranteed under mild regularity conditions.

In the following we assume that there are constants $\kappa_0$, $\kappa_1$ and $\tau$ such that the tail of the error terms $\epsilon(s)$ is bounded for all $z\ge 0$ as
\begin{align}\label{tailCondition}
	\p( |\epsilon(s)| > z ) \le \kappa_0 \exp\left( -\kappa_1\,z^{\tau} \right), \text{ for some } \kappa_0,\kappa_1, \tau > 0.
	\end{align}
The error terms are sub Gaussian if \eqref{tailCondition} is satisfied for a $\tau \ge 2$. \cite{kohler2008nonlinear} derives a rate of convergence for sub Gaussian error terms in a regression model with i.i.d.\ data. We shall in the following give a bound for the general case where $\tau > 0$.


\section{Nonlinear Hard Thresholded Regression}\label{Section_NonlinRegWavelets}
We outline in detail the estimation procedure in the first part of this section. In the second part, we present the results.

\subsection{The Estimation Procedure}
We estimate the conditional mean function $m$ with the nonlinear hard thresholding estimator which belongs to the class of sieve estimators (cf. \cite{grenander1981abstract}). Let there be given a sequence of linear spaces $\cF_k$ with functions $g_i : \R^d\rightarrow\R$ such that,
\begin{align}\label{LinSpaceGeneral}
		\cF_k = \left \{ \sum_{i=1}^{K^*} a_i g_i: a_i \in \R \right\}, \quad \cF_k \subseteq \cF_{k+1} \text{ and } \bigcup_{k\in\N_+} \cF_k \text{ is dense in } L^2(\mu).
\end{align}
The dimension of the linear spaces $K^*\in\N_+$ depends on the index $k\in\N_+$ and is defined for each particular case.\\

In particular, in this paper we shall consider a sequence $\mathcal{F}_k$ generated by Haar wavelets in $d$-dimensions, cf. \cite{benedetto1993wavelets}: let $\xi_0=1_{[0,1)}$ be a Haar scaling function on the real line $\R$ together with the Haar mother wavelet $\xi_1=1_{[0,1/2)}-1_{[1/2, 1)}$. Define the diagonal matrix $M$ by $M \coloneqq 2 \text{ diag}(1,\ldots,1)\in\R^{d\times d}$. Set $|M| \coloneqq \operatorname{det}(M) = 2^d$. Denote the mother wavelets as pure tensors by 
\begin{align}\label{MultidimHaarWaveletsMother}
		\Psi_v \coloneqq \xi_{v_1} \otimes \ldots \otimes \xi_{v_d} \text{ for } v \in \{0,1\}^d\setminus \{0\}.
\end{align}
The scaling function is given by
\begin{align}\label{MultidimHaarWaveletsFather}
	\Phi \coloneqq \Psi_0 \coloneqq \bigotimes_{i=1}^d \xi_0.
	\end{align}
One can show that $\Phi$ and the $\R$-linear spaces $V_j \coloneqq \left\langle \Phi (M^j \cdot - \gamma): \gamma \in \Z^d \right\rangle$ generate a multiresolution analysis of $L^2(\lambda^d)$ and that the functions $\Psi_v$, $v\neq 0$, form an orthonormal basis in the sense that
\begin{align*}
		&L^2(\lambda^d) = V_0 \oplus \left( \bigoplus_{j\in\N} W_j \right) = \bigoplus_{j\in\Z} W_j \\
		&\qquad\qquad \text{ where } W_j = \left\langle\, |M|^{j/2}\, \Psi_v\left( M^j\darg - \gamma \right): \gamma \in \Z^d,\, v \in \{0,1\}^d \setminus \{0\} \, \right\rangle.
\end{align*}
Furthermore, this wavelet family is dense in $L^2(\tilde{\mu})$ for each probability measure $\tilde{\mu}$ on $\cB(\R^d)$, (a proof is given in \cite{krebs2016nonparametric}).\\

We come to the construction of the function space from \eqref{LinSpaceGeneral}: let $j_0$ and $j_1$ be two integers. Here $j_0$ is the coarsest resolution and $j_1$ is the finest resolution for the wavelets in the function space. The finest resolution $j_1 $ is a function of $k$ and increases with the sample size, we precise this below. Define the rescaled and shifted father and the mother wavelets by
$$\Phi_{j_0,\gamma} \coloneqq |M|^{j_0/2}\, \Phi\left(M^{j_0}\cdot-\gamma\right) \text{ and } \Psi_{v,j,\gamma} \coloneqq |M|^{j/2}\, \Psi_v\left(M^{j}\cdot-\gamma\right), $$
for $v=1,\ldots,2^d-1$ and $j\ge j_0$. Now, consider the $\R$-linear spaces 
\begin{align}\begin{split}\label{HaarWaveletsLinSpace}
		\cF_k &\coloneqq \Bigg\{ \sum_{\gamma\in A_{j_0,k} } a_{j_0,\gamma}\, \Phi_{j_0,\gamma} + \sum_{v=1}^{2^d-1} \sum_{j=j_0}^{j_1(k)-1} \sum_{\gamma \in A_{j,k} } b_{v,j,\gamma} \, \Psi_{v,j,\gamma}: a_{j_0,\gamma}, b_{v,j,\gamma} \in \R,\\
		&\qquad\qquad\qquad j=j_0,\ldots,j_1 - 1, v=1,\ldots,2^d-1, \gamma\in A_{j,k} \Bigg\},
\end{split}\end{align}
together with the index sets $A_{j,k} \subseteq \Z^d$. These index sets are given by 
$$A_{j, k} \coloneqq \left\{ -2^{j-j_0} w_k ,\ldots, 2^{j-j_0} w_k - 1 \right\}^d $$
for a non decreasing sequence $(w_k:k\in\N_+) \subseteq \N_+$. Note that with the definitions of the $A_{j,k}$ we have that the father wavelets cover the cube $2^{-j_0} [-w_k, w_k)^d$ which has dyadic edges for each $j \ge j_0$, i.e.,
\begin{align}\label{CummulativeSupport}
	\bigcup_{\gamma\in A_{j,k} } \supp{\Phi_{j,\gamma}} = 2^{-j_0} [-w_k,w_k)^d =: D_k, \quad j\ge j_0.
	\end{align}
If the distribution of the $X(s)$ is bounded, it suffices to take a constant sequence such that $\bigcup_{\gamma \in A_{j_0, k}} \supp{\Phi_{j_0,\gamma}}$ covers the domain of the distribution. Otherwise, we choose $w_k$ as increasing; we precise this in the subsequent theorems. \\

In the following, we give a construction of $d$-dimensional wavelets which are orthonormal w.r.t.\ the empirical measure for $d\ge 1$. Therefore, we first consider in a short example on how design adapted wavelets can be constructed in one dimension, cf. \cite{FrankeSachs}. Let $I=[a,b) \subseteq \R$ be a finite interval and let $I_{j,1},\ldots,I_{j,2^j}$ be a dyadic partition of $[a,b)$ for each $j\ge 0$, i.e., 
$$I_{j,\gamma} = \left[a+\frac{b-a}{2^j}(\gamma-1), a+\frac{b-a}{2^j} \gamma \right ) \text{ for } \gamma = 1,\ldots,2^j.$$
Then, define the father and the mother wavelets for $j\ge 0$ and $\gamma \in \left\{1,\ldots,2^j \right\}$ as
\begin{align*}
		&\phi_{j,\gamma} \coloneqq \mu_n(I_{j,\gamma})^{-1/2}\, \I\{I_{j,\gamma}\} \text{ and }  \\
		&\psi_{j,\gamma} \coloneqq \mu_n(I_{j,\gamma})^{-1/2} \left( \mu_n( I_{j+1,2\gamma+1})^{1/2} \phi_{j+1,2\gamma} - \mu_n( I_{j+1,2\gamma})^{1/2} \phi_{j+1,2\gamma + 1}		\right). 
\end{align*}
The $\phi_{j,\gamma}$ and the $\psi_{j,\gamma}$ are both orthonormal. Additionally the mother wavelets $\psi_{j,\gamma}$ are balanced w.r.t.\ the empirical measure, which means that $\int_{\R} \psi_{j,\gamma} \intd{\mu_n} =0$ for a sample $X_1,\ldots,X_n \subseteq \R$.\\
Next, we outline in detail how to construct orthonormal and balanced wavelets in $L^2(\mu_n)$ for higher dimensional data. However, these wavelets no longer fulfill the usual scaling equations which are satisfied in the case of the Lebesgue measure because the empirical measure $\mu_n$ on $\R^d$ is not a product measure if $d>1$ and a partition of $\R^d$ into Cartesian products of intervals in general does not satisfy that each partition element contains the same number of sample points. \\
As the father wavelets partition $\R^d$, we can use the following construction: let there be given a cube $C=C_1\times\ldots\times C_d$ where the $C_i=[a_i,b_i) \subseteq \R$ are finite intervals. Let the family of standard Haar wavelets $\{ \Psi_k: k\in\{0,1\}^d \}$ from Equations~\eqref{MultidimHaarWaveletsMother} and \eqref{MultidimHaarWaveletsFather} be translated and rescaled such that $\Phi = \Psi_0 = \I\{C\}$.\\
Let $\iota$ be an enumeration of the index set $\{0,1\}^d $ as follows, 
\[
		\iota(0) = (0,\ldots,0), \quad \iota(1) = (1,0,\ldots,0),\quad \iota(2) = (0,1,0,\dots,0), \ldots, \iota\left(2^d-1 \right) = (1,\ldots,1).
\]
The enumeration $\iota$ defines a linear ordering of the $2^d$ basis functions, e.g., the first basis function is the one with index $\iota(1)$, the last one is the one with index $\iota\left(2^d -1 \right)$.\\
Define the first function on $C$ by $f_0 \coloneqq \mu_n(C)^{-1/2} \I\{C\}$. Denote the left half of an interval $C_i$ by $C_i^L \coloneqq [a_i, (a_i+b_i)/2)$ and the right half by $C_i^R \coloneqq [(a_i+b_i)/2, b_i)$. Furthermore set $C_{-i} \coloneqq \prod_{ j\neq i } C_j$. Define the second function which is the first balanced (and orthonormal) wavelet as
\begin{align}\label{EmpiricalOrthonormalization}
	f_1 &\coloneqq \sqrt{ \frac{ \mu_n \left( C_1^R  \times C_{-1}\right) }{ \mu_n\left( C_1^L \times C_{-1} \right)  \mu_n(C)} } 1\left\{ C_1^L \times C_{-1} \right\} - \sqrt{ \frac{\mu_n\left( C_1^L \times C_{-1} \right) }{  \mu_n \left( C_1^R \times C_{-1} \right) \mu_n(C) } } 1\left\{  C_1^R \times C_{-1} \right\}.
\end{align}
Assume that w.r.t.\ $\mu_n$ the functions $f_0,\ldots,f_u$ are orthonormal and that $f_1,\ldots,f_u$ are additionally balanced w.r.t.\ $\mu_n$ for $1\le u \le 2^d-1$. Then the function $\tilde{f}_{u+1}$ which one obtains from the Gram-Schmitt rule
\begin{align*}
		\tilde{f}_{u+1} \coloneqq \Psi_{ \iota(u+1) } - \scalar{\Psi_{ \iota(u+1) } }{f_u}_n f_u - \ldots - \scalar{\Psi_{ \iota(u+1) } }{f_1}_n f_1 - \scalar{\Psi_{\iota(u+1)} }{f_0}_n f_0
\end{align*}
is balanced (and orthogonal to $\langle f_0,\ldots,f_u\rangle$). If $\tilde{f}_{u+1}$ is not zero w.r.t $\mu_n$, normalize this function and obtain $f_{u+1}$. If one repeats this step until $u=2^d-1$, one obtains $\widetilde{K} \le 2^d-1$ balanced and orthonormal wavelets $f_1,\ldots,f_{\widetilde{K}}$. The index $\widetilde{K}$ gives the total number of wavelets and can be smaller than $2^d-1$ because some wavelets can be zero in $L^2(\mu_n)$. \\
Note that  each of these functions is constant on a subcube 
$$C_S = \bigtimes_{i=1}^d C_i^{S_i}$$ 
with $S \in \{L,R\}^d$. Hence, executing the same procedure on $C_S$, one obtains a new set of $\widetilde{K}_S \le 2^d-1$ orthonormal and balanced wavelets, call them $g_1,\ldots,g_{\widetilde{K}_S}$. The $g_i$ are orthogonal to $f_1,\ldots,f_{\widetilde{K} }$, too, because the latter are constant on the domain of the $g_i$. Consequently, if one repeats this procedure (ad infinitum) one can construct a wavelet family on $C$ (and its dyadic subcubes) which is orthonormal.\\

In the last step, we consider a fixed resolution index $j_0$ and choose the partition of $\R^d$ which is given by the collection of cubes of $\{ C_\gamma, \gamma\in \Z^d\}$, where each $C_\gamma = 2^{-j_0} [\gamma, \gamma + e_N)$ is a dyadic subcube. We choose the wavelet family on entire $\R^d$ which is orthonormal w.r.t.\ the empirical measure $\mu_n$. We denote this family for a fixed $j_0$ by
$$\left\{f_{0,j_0,\gamma}, f_{v,j,\gamma}: v=1,\ldots,2^d-1, j\ge j_0, \gamma\in\Z^d \right\}.$$
In particular, the following equality for the father wavelets is true
$$f_{0,j_0,\gamma} = \I\left\{ 2^{-j_0}[ \gamma,\gamma+e_N) \right\} \mu_n\left( 2^{-j_0}[ \gamma,\gamma+e_N) \right)^{-1/2} =   \mu_n\left( 2^{-j_0}[ \gamma,\gamma+e_N) \right)^{-1/2} / 2^{j_0 d /2} \; \Phi_{j_0,\gamma}.$$
Then, upon replacing $\Phi_{j_0,\gamma}$ and $\Psi_{v,j,\gamma}$ by the orthonormal counterparts $f_{0,j_0,\gamma}$ and $f_{v,j,\gamma}$, we have for the function spaces from Equation \eqref{HaarWaveletsLinSpace} the equality
\begin{align}\label{HaarWaveletsLinSpace2}
		\cF_k &= \left\langle f_{0,j_0,\gamma}, f_{v,j,\gamma}: v=1,\ldots,2^d-1, j=j_0,\ldots,j_1 - 1, \gamma \in A_{j,k} \right\rangle  \text{ in } L^2(\mu_n)\quad a.s.
\end{align}
Thus, this space is spanned by at most $K^*(k) = (2\cdot 2^{j_1(k)-j_0} \, w_k )^d$ wavelets.\\
\cite{kohler2008nonlinear} constructs an alternative orthonormal basis with Haar wavelets, which has as well the property that the functions are balanced w.r.t.\ $\mu_n$, however, each function vanishes on a larger set than our corresponding function.

We introduce some extra notation which will help us to precise the rates of convergence given below. We have from \eqref{CummulativeSupport} and by the definition of the index set $A_{j,k}$ that the support of the father wavelets (and of the mother wavelets) from the function space $\cF_k$ is $D_k = \bigcup_{\gamma\in A_{j_0,k}} \supp{ \Phi_{j_0,\gamma}} = 2^{-j_0}[-w_k,w_k)^d$. We define inductively for $D_k$ sets of partitions $\prod_u$ for $1\le u \le u_{max}$, where the maximal index $u_{max}$ is given by
\begin{align}\label{uMax}
	u_{max} \coloneqq 1+ (2w_k)^d [ (2^{d(j_1-j_0)}-1)/(2^d-1) ].
	\end{align}
The reason for this definition is explained in the next lines. Each $\prod_u$ contains a family of partitions of $D_k$ as follows: set $\prod_1 \coloneqq \{ \supp{\Phi_{j_0,\gamma}}: \gamma \in A_{j_0,k} \}$. Hence, $\prod_1$ contains a single partition. Let $\prod_1,\ldots,\prod_u$ be constructed. Then $\pi \in \prod_{u+1}$ if and only if there is a partition $\pi'$ of $D_k$ in $\prod_u$ and a dyadic cube $C$ such that
\[
	\pi = \pi' \setminus \{ C\} \cup \left\{ C_1^{u_1} \times \ldots C_d^{u_d} : u_i \in \{L,R\}  \,\right\}
	\]
and each element of the partition $\pi$ can be written as $ E_1\times \ldots \times E_d$ where the $E_i$ are intervals of a length greater or equal than $2^{- j_1 }$. This means $\prod_{u+1}$ consists of all partitions of $\prod_u$ which are refined by partitioning exactly one element into $2^d$ equivolume cubes and each partition element has a diameter w.r.t.\ the $\infty$-norm of at least $2^{- j_1 }$. Note that $u_{max}$ is the maximal index such that $\prod_{u+1}$ contains finer partitions than $\prod_u$ for all $u< u_{max}$. Additionally, $\prod_{u_{max}}$ contains a single partition, namely,	$\prod     _{u_{max}} = \{ \supp{\Phi_{j_1,\gamma}}: \gamma \in A_{j_1,k} \}$.\\
We denote the functions which are constant w.r.t.\ a partition $\pi \in \prod_u$ for some $1\le u \le u_{max}$ by $\cF_c \circ \pi$. This finishes our construction of the function space $\cF_k$ with orthonormal wavelets.\\

In the following, we consider again the more general case that the collections $\cF_k$ are given as the linear span of $K^*$ functions $g_1,\ldots,g_{K^*}$ as in \eqref{LinSpaceGeneral}. We apply the Gram-Schmitt orthonormalization algorithm to the functions $g_1,\ldots,g_{K^*}$ as in \eqref{EmpiricalOrthonormalization}. Thus, we obtain as in the special case for the Haar wavelets a set of functions $f_1, \ldots, f_{\tilde{K} } \in \cF_k$ such that
$$\cF_k = \langle f_1,\ldots,f_{ \tilde{K} } \rangle \text{ in } L^2(\mu_n) \quad a.s.$$
Note that some functions might be zero w.r.t.\ the empirical measure, so $\tilde{K} \le K^*$ $a.s.$ Furthermore, these (random) functions satisfy per construction for each $1\le u \le K^*$ the relation
\[
		\langle f_1,\ldots, f_u \rangle = \langle g_1,\ldots, g_u \rangle \text{ in } L^2(\mu_n),
\]
For such an orthonormal system we consider the following estimation procedure: set $J\coloneqq \{ 1,\ldots,\tilde{K} \}$ and define the $|I_n|\times \widetilde{K}$-matrix which contains in column $j$ the function $f_j$ evaluated at the sample data $X(s)$, $s\in I_n$:
$$\cD \coloneqq \Big( \, f_j (X(s) ) \, \Big)_{s\in I_n, j\in J}.$$
By construction $|I_n|^{-1} \cD^T \cD = I$ is the identity matrix. Then define the estimates of the coefficients of the regression function as the projection of $Y$ over $f_j$ in the empirical norm
\begin{align*}
		a^*  = \frac{1}{|I_n|} \cD^T Y \Leftrightarrow a^*_j = \frac{1}{|I_n|} \sum_{s\in I_n} f_j (X(s)) Y(s).
\end{align*}
Let $(\lambda_k: k \in\N) \subseteq \R_+$ be a sequence which converges to zero, we call this sequence the (hard) thresholding sequence. Define the nonlinear thresholded estimator $m_{k,J^*}$ by
\begin{align*}
	J^{*} \coloneqq \left\{ j\in J: |a_j^*| > \lambda_k \right\} \quad \text{and} \quad m_{k, \widehat{J}} \coloneqq \sum_{j\in \widehat{J}} a^*_j f_j  \quad \text{for} \quad \widehat{J} \subseteq J.
\end{align*}
Hence, $m_{k,J^*} = \sum_{j\in J^*} a^*_j f_j$ is the linear combination which consists of exactly those functions $f_j$ whose estimated contribution in terms of $a^*_j$ exceeds the thresholding value $\lambda_k$.
One can show with elementary reasoning that hard thresholding corresponds to $L^0$-penalized least squares, cf. \cite{kohler2003nonlinear}: define the penalizing term for a subset $\widehat{J}$ of $J$ as $pen_k ( \widehat{J} ) \coloneqq \lambda_k^2\, | \widehat{J}|$. Then, $m_{k,J^*}$ satisfies the relation
\begin{align*}
		\frac{1}{|I_n|} \sum_{s\in I_n} |m_{k,J^{*}}(X_s) - Y_s|^2 + pen_k(J^{*}) = \min_{ \widehat{J} \subseteq J} \left\{ \frac{1}{|I_n|}  \sum_{s\in I_n} |m_{k, \widehat{J}}(X_s) - Y_s|^2 + pen_k( \widehat{J}) \right\}.
\end{align*}
In the following, we write for short $m_k \coloneqq m_{k,J^{*}}$ for the minimizing function. Define for $L>0$ the truncation operator as $T_L(y) \coloneqq \max( \min(y,L), -L)$. Let  $\{\beta_k: k \in\N_+\}$ be a real-valued, non decreasing truncation sequence which converges to infinity. In order to make the estimator robust in regions of $\R^d$ where the data $\{X(s): s\in I_n \}$ is sparse, we consider the truncated hard thresholding least-squares estimator
\begin{align}
		\hat{m}_k \coloneqq T_{\beta_k} \, m_k. \label{lsqIII}
\end{align}
Furthermore, for a function $f\in \cF_k$, which has a unique representation w.r.t.\ the orthonormalized functions $\{ f_1,\ldots,f_{\tilde{K}} \}$ in $L^2(\mu_n)$, we consistently extend the definition of the penalizing term and write $pen_k(f)$ for the number of nonzero coefficients in this representation multiplied by $\lambda_k^2$. Then, $pen_k(f)$ is stochastic and bounded by $K^*(k)\,\lambda_k^2$.

\subsection{Consistency and Rate of Convergence - Results}
The estimator $\hat{m}_k$ is consistent under the conditions in the next theorem:

\begin{theorem}[Consistency]\label{ConsisNonlinReg}
Assume the random field $\{ X(s),Y(s) : s\in\Z^N\}$ satisfies Condition~\ref{regCond0} and Equation \eqref{lsqI} for some functions $m$ and $\varsigma$ in $L^2(\mu)$. Let $n(k)$ be an increasing sequence in $\N^N$ which defines the observation domain and which satisfies \eqref{Cond_Sequence} and \eqref{Cond_Sequence2}. Let the function classes be given as linear spaces which are as in \eqref{LinSpaceGeneral}. In particular, if the spaces are given by \eqref{HaarWaveletsLinSpace} and \eqref{HaarWaveletsLinSpace2} and if the distribution of the $X(s)$ is unbounded, let $\lim_{k\rightarrow \infty} w_k = \infty$.\\
If both
\[
		 K^*(k) \, \lambda_k^2 \rightarrow 0 \text{ and } 	\beta_k^2 \, K^*(k) \, \log \beta_k \prod_{i=1}^{N} \log n_i(k) \Bigg/ \left(		\prod_{i=1}^N n_i(k) \right)^{1/(N+1)} \rightarrow 0 \text{ as } k \rightarrow \infty,
\]
then the estimator $\hat{m}_k$ is weakly universally consistent in the sense that
\[
	\lim_{k\rightarrow \infty} \E{ \int_{\R^d} |\hat{m}_k - m|^2 \,\intd{\mu} } = 0.
\]
Note that if the function spaces are given by \eqref{HaarWaveletsLinSpace} and \eqref{HaarWaveletsLinSpace2}, then $K^*(k) = (2^{j_1(k) - j_0(k) +1} w_k)^d$.\\
Moreover, if in addition $\{ Y(s): s\in\Z^N \}$ is stationary and if in addition for some positive $\delta > 0$
\[
	\beta_k^2 \, (\log k)^{1+\delta} \prod_{i=1}^{N}  \, \log n_i(k) \Bigg/ \left(		\prod_{i=1}^N n_i(k) \right)^{1/(N+1)} \rightarrow 0 \text{ as } k \rightarrow \infty,
\]
then the nonlinear wavelet estimator $\hat{m}_k$ is strongly universally consistent in the sense that
\[
	\lim_{k\rightarrow \infty} \int_{\R^d} |\hat{m}_k - m|^2 \,\intd{\mu}  = 0 \quad a.s.
\]

\end{theorem}

Under the condition that the error terms are exponentially decreasing as in \eqref{tailCondition}, we can further derive a rate of convergence theorem both for the general linear space and for the wavelet system.

\begin{theorem}[Rate of convergence]\label{NonlinRateOfConvExpDecr}
Let Condition \ref{regCond0} be satisfied for the random field $(X,Y)$. Furthermore, let the regression function $m$ and the conditional variance function $\varsigma^2$ be essentially bounded. Set $B \coloneqq \norm{m}_{\infty}$ and $\beta_k :\equiv B$ for all $k\in\N_+$. Assume the error terms $\epsilon(s)$ fulfill the tail condition~\eqref{tailCondition} for some parameters $\kappa_0,\kappa_1$ and $\tau$. Assume additionally the thresholding sequence $\lambda_k$ and the growth of the basis functions satisfy the relations
\[
		K^*(k) \, \lambda_k^2 \rightarrow 0 \text{ and } \left(	K^*(k) \right)^{1+2/\tau} \left(	\prod_{i=1}^N \log n_i(k)	\right)^{3+4/\tau} \Bigg/\,  |I_{n(k)}|^{1/(N+1)} \rightarrow 0 \text{ as } k \rightarrow \infty.
\]
Consider a sequence of function spaces satisfying \eqref{LinSpaceGeneral}. Then there is a constant $C\in\R_+$ which only depends on $B$, $\norm{\varsigma}_{\infty}$, the lattice dimension $N$, the bound on the mixing coefficients and the parameters of the tail distribution of the error terms such that the $L^2$-error satisfies for all $k\in\N_+$
\begin{align}\begin{split}\label{EqConvRate}
	\E{ \int_{\R^d} |\hat{m}_k - m |^2\, \intd{\mu} } 	&\le 4 \min_{1\le u \le K^*(k) } \left\{ \lambda_k^2 u + \inf_{ \substack{ f\in\cF_k:\\ f = \sum_{i=1}^u a_i g_i }} \int_{\R^d} | f-m| ^2\,\intd{\mu} \right\}	\\
	&\quad + C\, \left(	K^*(k) \right)^{1+2/\tau} \left(	\prod_{i=1}^N \log n_i(k)	\right)^{3+4/\tau} \Bigg/\, |I_{n(k)}|^{1/(N+1)} .
\end{split}\end{align}
In the particular case the sequence of function spaces satisfy \eqref{HaarWaveletsLinSpace} and \eqref{HaarWaveletsLinSpace2}, the bound from \eqref{EqConvRate} can be refined as
\begin{align}
		\E{ \int_{\R^d} |\hat{m}_k - m |^2\, \intd{\mu} }  &\le 4 \min_{1\le u \le u_{max} }  \left\{	\lambda_k^2 \left( (2^d-1)(u-1) + (2w_k)^d \right) 	+ \min_{\pi \in \prod_u} \inf_{f \in \cF_c \circ \pi } \int_{\R^d} |f-m|^2\,\intd{\mu} \right\} \nonumber \\
		&\quad + C\, \left(	 2\cdot 2^{j_1-j_0} w_k)\right)^{d(1+2/\tau)} \left(	\prod_{i=1}^N \log n_i(k)	\right)^{3+4/\tau} \Bigg/\, |I_{n(k)}|^{1/(N+1)}. \nonumber
\end{align}
\end{theorem}

Note that we consider in \eqref{EqConvRate} the infimum over the linear combinations with the deterministic functions $g_i$, this means that the approximation error, i.e., the first term in \eqref{EqConvRate}, is deterministic. The same is true for the special case of the Haar wavelet system.\\
Under the more severe restriction of a bounded regression function $m$ and bounded error terms, the rate of convergence of $\hat{m}_k$ can be improved. We only state the result for the general function spaces from \eqref{LinSpaceGeneral}. The application to the orthonormal wavelet system from \eqref{HaarWaveletsLinSpace} and \eqref{HaarWaveletsLinSpace2} is straightforward.

\begin{theorem}\label{RateConvNonlinThrm}
Let the conditions of Theorem \ref{NonlinRateOfConvExpDecr} be satisfied. Additionally, let the error terms $\epsilon(s)$ be essentially bounded by $B'\in\R_+$, i.e., $|\epsilon(s)|\le B'$ for all $s\in\Z^N$. Set $\beta_k =  B+\norm{\varsigma}_{\infty} B'$. Then there is a constant $C\in\R_+$ which only depends on $B$, $B'$, $\norm{\varsigma}_{\infty}$, the lattice dimension $N$ and the bound on the mixing coefficients such that the estimator $\hat{m}_k$ satisfies for all $k\in\N_+$
\begin{align}\begin{split}\label{EqConvRateV2}
		\E{ \int_{\R^d} | \hat{m}_k - m|^2\, \intd{\mu} }	&\le 4 \min_{1\le u \le K^*(k) }  \left\{ \lambda_k^2 u + \inf_{ \substack{ f\in\cF_k:\\ f = \sum_{i=1}^u a_i g_i }} \int_{\R^d} | f-m| ^2\,\intd{\mu} \right\}	\\
		&\quad + C\, K^*(k) \, \left(	\prod_{i=1}^N \log n_i(k)	\right)^3 \,\Big/\, |I_{n(k)}|^{1/(N+1)} .
\end{split}\end{align}
\end{theorem}

We see that the estimation error which is the second term on the r.h.s. of \eqref{EqConvRate} resp. \eqref{EqConvRateV2} decreases at a rate of $\left(	K^*(k) \right)^{1+2/\tau} \left(	\prod_{i=1}^N \log n_i(k)	\right)^{3+4/\tau} \big/\, |I_{n(k)}|^{1/(N+1)}$ where $\tau$ converges to infinity in the case of bounded error terms, note that $K^*(k)$ depends on the data dimension $d$. Hence, an increase in the tail parameter $\tau$ influences the rate positively because the conditional variance is reduced. In particular, for sub Gaussian error terms, this means that the estimation error decreases at a rate of at least the squared number of basis functions divided by the sample size raised to $1/(N+1)$ modulo some logarithmic terms which cannot be avoided. Below, we shall give a more detailed discussion on the influence of the lattice dimension $N$, in particular, in the light of the findings which have been made in other research articles.\\
The approximation error depends on the smoothness of the regression function. If we impose smoothness assumptions on $m$, we can derive a rate of convergence. We give two examples of application for the isotropic Haar basis: we choose an $(L,r)$-Hölder continuous regression function $m$ which means there is an $L\in\R_+$ and $r \in (0,1 ]$ such that
\[
	|m(x) - m(y)| \le L \norm{ x-y }^r_{\infty} \text{ for all } x,y \text{ in the domain of }m.
	\]
\begin{corollary}[Rate of convergence for Hölder continuous functions]\label{RateOfConvergenceExpErrorsHaarBasis}
Let the conditions of Theorem \ref{NonlinRateOfConvExpDecr} be satisfied for the function spaces which are constructed from the orthonormal wavelet system as given in \eqref{HaarWaveletsLinSpace} and \eqref{HaarWaveletsLinSpace2}. Let $m$ be $(L,r)$-Hölder continuous. Furthermore, assume that the Euclidean norm of $X$ is integrable w.r.t.\ the probability measure $\p$ for some $\gamma \in \R_+$, i.e., $\norm{X}_{2}^{\gamma}$ .  Let $C_0, C_1, C_2 \in\R_+$ be constants. Set
\begin{align*}
	& \widetilde{R}(n) =		\left(\prod_{i=1}^N \log n_i \right)^{ 3 + 4/\tau} \Big/ |I_n|^{1/(N+1)}
	\end{align*}
and define the parameters as
\begin{align*}
	& w_{k} = \floor*{ C_0 \exp\left\{ - \frac{ 2r}{2rd(1+2/\tau) + \gamma (2r +d(1+2/\tau)) } \log \tilde{R}(n(k)) \right\} } ,\\
	&\lambda_k^2 = C_1\, \exp\left\{ \frac{ 2rd + (2r + d) \gamma }{ 2rd(1+2/\tau) + \gamma(2r + d(1+2/\tau)) } \log  \widetilde{R}(n(k))  \right\} \\
	&\text{ and } j_1(k) = \floor*{ C_2 - \frac{\gamma}{2r d(1+2/\tau)  + \gamma(2r + d(1+2/\tau) ) } \frac{ \log \widetilde{R}(n(k)) }{ \log 2 } }  .
	\end{align*}
Then the rate of convergence is at least
\begin{align*}
		\E{ \int_{\R^d} | \hat{m}_n - m|^2\, \intd{\mu} } & = \cO\left( \widetilde{R}(n(k))^{ 2r\gamma \,/\, [2rd(1+2/\tau) + \gamma(2r + d(1+2/\tau)) ] }	\right).
\end{align*}
If the distribution of $X$ is bounded and if
\begin{align*}
		&\lambda_k^2 = C_1\, \exp\left\{	 \frac{2r+d}{2r+d(1+2/\tau)} \log  \widetilde{R}(n(k))  \right\} \\
		&\text{ and } j_1(k) = \floor*{C_2 - \frac{1}{2r + d(1+2/\tau) } \frac{ \log \widetilde{R}(n(k)) }{ \log 2 } },
\end{align*}
the estimator achieves a rate of at least $\cO\left( \widetilde{R}(n(k))^{ 2r \,/\,[2r + d(1+2/\tau) ]   } \right)$.
\end{corollary}
\begin{proof}
The proof is achieved by computing the approximation error, we choose $j_0 = 0$ as the roughest resolution: there is a function $f \in \cF_k$ (from Equations \eqref{HaarWaveletsLinSpace} and \eqref{HaarWaveletsLinSpace2}) which is piecewise constant on dyadic $d$-dimensional cubes of edge length $2^{- j_1 }$ with values 
\begin{align*}
	&f(x) = m\left( x^*(\gamma) \right) \text{ for } x \in \left[ (\gamma_1,\ldots,\gamma_d)/2^{j_1}, ((\gamma_1,\ldots,\gamma_d) + e_N)/2^{j_1} \right) \\
	&\qquad\qquad \text{ and } x^*(\gamma) \in \left[ (\gamma_1,\ldots,\gamma_d)/2^{j_1}, ((\gamma_1,\ldots,\gamma_d) + e_N)/2^{j_1} \right) \cap dom(m)  \\
	&\qquad\qquad \qquad\qquad \text{ where } \gamma_i \in \Z \text{ for }i=1,\ldots,d. 
	\end{align*}
In case of an unbounded distribution of $X$, the domain of $f$ is the cube $[-w_k, w_k)^d$ and in case of a bounded distribution, $f$ is supposed to be defined on the entire domain of $X$. The approximation error is at most
\[
	\int_{B} |f-m|^2\,\intd{\mu} \le \sup_{\text{dom}(f) } |f-m|^2 + \norm{m}_{\infty}^2 \p( \norm{X}_{\infty} \ge w_k ) \le L^2\, 2^{-2\,j_1(k)\,r } + \norm{m}_{\infty}^2 \E{ \norm{X}_{\infty}^{\gamma} } \,\Big/ \, w_k^{\gamma} .
\]
If the distribution is bounded the second term on the RHS in the last expression is zero. The growth rates of $j_1$, $\lambda_k$ and $w_k$ equalize the asymptotic rates of the error terms in both cases.
\end{proof}

Corollary~\ref{RateOfConvergenceExpErrorsHaarBasis} reveals that the lattice dimension $N$ and the data dimension $d$ influence the rate of convergence negatively whereas the parameters $\tau$ and $r$ have a positive influence. The role of the lattice dimension $N$ is discussed below. The effect of the data dimension $d$ is the well-known curse of dimensionality. The positive impact of the parameter $\tau$ is clear because an increase in $\tau$ controls the tail growth of the error terms. Similarly, an increase in $r$ means a smoother regression function which can be better approximated by a piecewise constant function.\\
We discuss in detail the influence of the lattice dimension $N$. Therefore, we first sum up classical results for i.i.d.\ data which are the most comparable to the case where $N=1$: \cite{stone1982optimal} shows that the optimal rate of convergence for $(L,r)$-H{\"o}lder continuous functions is in $\cO\left( n^{-2r/(2r+d) } \right)$. In the classical case for a Hölder continuous regression function which is defined on a bounded domain and for an i.i.d.\ sample of $n$ observations the  $L^2$-risk decreases essentially at a rate of $n^{-2r/(2r+d) }$ times a logarithmic factor: \cite{kohler2008nonlinear} considers a multivariate set-up for an i.i.d.\ sample $\{ X_1,\ldots,X_n \} \subseteq [0,1]^d$, a bounded regression function $m$ and sub Gaussian error terms. This corresponds to our scenario of bounded $X(s)$ and a decay rate for the error terms of at least $\tau \ge 2$. For the hard thresholding sequence $\lambda_n = C \sqrt{ \log n/ n} $ Kohler obtains a rate of $( \log n / n)^{2r/(2r+d)}$ which is nearly optimal. \cite{FrankeSachs} investigate the soft thresholding estimator for adaptive wavelets in the case of one-dimensional data. They require the existence of a compactly supported one-dimensional density for the distribution of $X$ and the existence of all moments of the error terms. For an adaptive soft thresholding estimator and an i.i.d.\ sample, they investigate a similarly defined rate of convergence (w.r.t.\ the empirical distribution) which in $\cO\left(	(\log n / n)^{2r/(2r+1) }	\right)$. \cite{li2016nonparametric} investigates a wavelet estimator for a Besov function $m$ in a very similar regression model for spatial data $(X,Y)$ as we do. However, he additionally assumes that the distribution of the $X(s)$ admits a compactly supported positive density which is known. Given these restrictions, the rate of convergence is again optimal modulo a log-loss.\\
In all three cases this log-loss is the result of the increasing complexity of the sieves. \cite{birge1997model} consider penalized nonparametric density estimation with sieve estimators for i.i.d.\ data. They show that under appropriate penalizing assumptions the log-loss can be avoided for  a special class of Besov functions. \cite{baraud2001adaptive} study penalized nonparametric regression function estimation for a $\beta$-mixing observation sequence $( (X_i,Y_i): 1\le i\le n )$ where the regressors $X_i$ are multidimensional, identically distributed and admit a density w.r.t.\ the Lebesgue measure and $Y_i=m(X_i)+\epsilon_i$. As in the present model \eqref{lsqI} the error terms are assumed to be independent of the data $X_i$.
\cite{baraud2001adaptive} consider a rate of convergence w.r.t.\ the empirical 2-norm defined by the design points $X_1,\ldots,X_n$ and the structure of the approximation error differs somewhat. They achieve for the estimation error a rate which is in $\cO(n^{-1})$ under a certain requirement on the decay of the $\beta$-mixing coefficients. However, note that the assumption of $\beta$-mixing data is stricter than the assumption of $\alpha$-mixing data.\\
In the present case of dependent data on a lattice the rates are different for general $N$ without further restrictions: consider for the rest of the discussion the best case where the approximation error is zero. Firstly, let additionally the error terms and the support of the distribution of the $X(s)$ be both bounded. This implies that the number of basis functions $K^*(k)$ can be chosen as constant and we see from Equations \eqref{EqConvRate} and \eqref{EqConvRateV2} that the best rate which is possible is in $\cO\Big( \left(	\prod_{i=1}^N \log n_i(k)	\right)^3 \,\Big/\, |I_{n(k)}|^{1/(N+1)}  \Big)$ independent of the data dimension $d$. If additionally in this case $N=1$, then for a sample of $n$ data points this means that the rate is at most $(\log n)^3 / n^{1/2}$. This corresponds to the findings made by \cite{modha1996minimum} who investigate the nonparametric regression model for stationary times series under minimal assumptions. In particular, they obtain for a one-dimensional times series $X=(X_k: k\in \Z)$ a rate of convergence of $\sqrt{\log n/ n}$.\\
Secondly, assume that we are still in the best case where the approximation error is zero but that the error terms are unbounded, our rate shows two further correction factors which come from the dependence relations: instead of $d$ we have $d (1+\tau)/\tau$ which is larger and additionally the exponent is multiplied by the standard correction factor which depends on the lattice dimension $N$. Consider the case $N=2$, let the distribution of the $X(s)$ be bounded and let $\tau \ge 2$: for the canonical sequence $n(k) \coloneqq k e_N$, we achieve a rate of $ \left( ( \log k)^6 / k^{2/3} \right)^{2r/(2r+2d)}$ for a sample of size $k^2$. The main reason for the sub-optimal rate is due to the fact that we allow for a variety of probability distributions of the $X(s)$ and of dependence structures within the lattice. The dependence in our model can spread in every dimension of the lattice, hence, observing data on an additional lattice dimension can become more and more redundant.  The technical reason for the worse rate of convergence of our estimator is the asymptotic decay of the Bernstein type inequalities for dependent data which we obtain and are presented in \ref{Appendix_ExpInequalities}. These can be compared with results of \cite{white1991some} and \cite{frankeBernstein}. The inequalities are derived under minimal assumptions on the distribution of the random field and consequently, only guarantee a slower rate which reflects the effective size of a dependent sample and not its nominal size.\\
We conclude this section with another example for piecewise $(L,r)$-Hölder continuous regression functions which means that there is a finite partition $\bigcup_{i=1}^S U_i$ of the domain of $m$ such that $m$ is $(L,r)$-Hölder continuous on each $U_i$, $1\le i \le S$.

\begin{corollary}\label{PiecewiseSmoothFunctions}
Let the same conditions be satisfied as in Corollary~\ref{RateOfConvergenceExpErrorsHaarBasis}. Additionally assume that $X$ takes values in a bounded domain $D = [-w , w ]^d$, for some $w\in \N_+$. The regression function $m$ is bounded by $B$ and is piecewise $(L,r)$-Hölder continuous such that for all $j \ge 0$ the condition
\begin{align}\label{PiecewiseHölderRegCond}
		\# \left\{\gamma \in  \{-2^{j}w,\ldots, 2^{j}w -1\}^d : m \text{ is not Hölder continuous on } 2^{-j}[\gamma,\gamma + e_n) \right\} \le C 2^{(d-1)j}
\end{align}
is satisfied for some constant $C$. Assume that the distribution of the $X(s)$ admits a density $g$ which is essentially bounded. Define for some $C_0,C_1\in\R_+$ the thresholding sequence and the resolution by
\begin{align*}
		\lambda_k^2 = C_0\, \exp\left\{ \frac{ 1\wedge 2r + d }{1\wedge 2r + d(1+2/\tau)} \log \widetilde{R}(n(k))   \right\},\quad j_1 = \floor*{C_1 - \frac{1}{1\wedge 2r + d(1+2/\tau)} \frac{\log \widetilde{R}(n(k))} {\log 2} }.
\end{align*}
Then, the $L^2$-error is in $\cO\left(	\widetilde{R}(n(k))^{ 1\wedge 2r \,/\, [ 1\wedge 2r + d(1+2/\tau)]}	\right)$.
\end{corollary}
\begin{proof}
The proof is similar: let there be given a resolution $j_1$ and fix the roughest resolution as $j_0 \coloneqq 0$ which correspond to a partition $\pi$ of the cube $[-w,w]^d$. Define $f$ as in the proof of Corollary~\ref{RateOfConvergenceExpErrorsHaarBasis}. Denote by $D_{dc}(j_1)$ the set of dyadic cubes of edge length $2^{-j_1}$ which contain points where $m$ is not continuous. Then the approximation error for a resolution up to $j_1$ is at most
\begin{align*}
		\int_D |f-m|^2\,\intd{\mu} \le L^2 2^{-2j_1 r} + \int_{D_{dc}(j_1)} |f-m|^2\,\intd{\mu} \le L^2 2^{-2j_1 r} + (2B)^2 \norm{g}_{\infty} 2^{-j_1 d} \, C 2^{(d-1) j_1 },
\end{align*}
here we use the regularity condition on the discontinuities from \eqref{PiecewiseHölderRegCond}. The definitions of $j_1$ and $\lambda_k$ equalize the individual error terms.
\end{proof}

Before we discuss this result, consider the requirement in Equation~\eqref{PiecewiseHölderRegCond}: if $d=1$, then \eqref{PiecewiseHölderRegCond} requires the number of discontinuities to be finite. Next, let $d\ge2$. We consider the boundary $\partial U$ of one such partitioning element $U \in \{ U_i: i=1,\ldots,S\}$. Therefore $\partial U$ is a finite union of smooth hypersurfaces, $\partial U = \bigcup_{t=1}^T H_t$, where each $H_t$ can be represented as the graph of a $C^1$-function: pick one such hypersurface $H$ which has w.l.o.g. the following representation and location in $\R^d$
\[
		H = \left\{ \left(x_{-d}, h(x_{-d}) \right): x_{-d} \in B \right\} \subseteq [0,1]^d, \text{ where } x_{-d} \coloneqq (x_1,\ldots,x_{d-1}) \text{ and } B \subseteq [0,1]^{d-1}
\]
and $h: B \rightarrow \R$ such that $\nabla h$ can be extended to a continuous function on $\overline{B}$. 
Let there be given the dyadic partition $\pi_j$ of the unit cube $[0,1]^d$ in $2^{dj}$ equivolume dyadic subcubes of edge length $2^{-j}$. Consider a partition element $\Box\in \pi_j$ which lies in the plane where the $d$-th dimension is zero  and intersects with $B$, i.e., $\Box\cap B \neq 0$. Then the number of partition elements $\widetilde{\Box}\in\pi_j$ which intersect with the image of $\Box \cap B$ under $h$ is bounded: indeed, use the "steepest ascent times longest path" approach which yields a maximal "height". Divide this number by the edge length of the cubes, this yields the approximate number of these partitioning elements. More formally and more precisely,
\begin{align*}
		& \Big| \{ \widetilde{\Box}\in \pi_j: \widetilde{\Box} \cap \{ (\Box \cap B) \times h (\Box \cap B) \} \neq \emptyset \} \Big|  \\
		&\qquad\qquad\qquad \le \left( 2^{-j }\,\sqrt{d-1}  \,\cdot\, \max_{x \in \overline{ B} } \norm{ \nabla h (x) }_{2 } \right) \Big/ 2^{-j} + 1 \le C
\end{align*}
for a constant $C$ which is independent of $j\in\N_+$. Hence, the total number of partition elements $\widetilde{\Box} \in\pi_j$ which intersect with $H$ is in $\cO( 2^{ (d-1) j } )$. Consequently, the total number of partition elements which intersect with $\partial U$ is in $\cO( 2^{(d-1)j})$ as required in \eqref{PiecewiseHölderRegCond}.\\
In light of this interpretation of the condition in \eqref{PiecewiseHölderRegCond}, Corollary~\ref{PiecewiseSmoothFunctions} illustrates that given there are discontinuities, an increase in the smoothness increases the rate of convergence only as long as $r < 1/2$, otherwise, if $r \ge 1/2$, the negative impact at the borders $\partial U_i$ is too prominent and dominates the approximating property on the parts of $D$ where $m$ is smooth.

\section{Simulation Examples}\label{Section_SimulationExamples}

In the first part of this section we introduce an algorithm to simulate a Markov random field $\{ (X(v),Y(v)): v\in V\}$ on a graph $G=(V,E)$ with a finite set of nodes $V$. In the second part we give a simulation example in the case where the graph is a finite regular lattice in two dimensions. 

\subsection{The Simulation Procedure}
We compute with the simulated random field the penalized least squares estimator as it is defined in Section~\ref{Section_NonlinRegWavelets} and compare its performance to the same estimator which is computed with an independent reference sample. Here an independent reference sample means a sample $(X_i,Y_i)$ of the same size and with the same marginal distributions as the random field $(X(v),Y(v))$, i.e., $\cL(X_i,Y_i) = \cL(X(v),Y(v))$ for $v\in V$ and $i=1,\ldots,|V|$.\\
The main idea for the simulation procedure dates back at least to \cite{kaiser2012} and is based on the concept of \textit{concliques} which has the advantage that simulations can be performed faster when compared to the Gibbs sampler; an introduction to Gibbs sampling offers \cite{bremaud1999markov}. We start with the definition of concliques:

\begin{definition}[Concliques, cf. \cite{kaiser2012}]
Let $G = (V,E)$ be an undirected graph with a countable set of nodes $V$ and let $C \subseteq V$. If all pairs of nodes $(v,w) \in C\times C$ satisfy $\{v,w\} \notin E$, the set $C$ is called a conclique. A collection $C_1,\ldots,C_n$ of concliques that partition $V$ is called a conclique cover; the collection is a minimal conclique cover if it contains the smallest number of concliques needed to partition $V$.
\end{definition}

Let now $\pspace$ be a probability space and let $(S,\fS)$ be a state space. Let $Y=\{ Y(v): v\in  V \}$ be a collection of $S$-valued random variables. Then the family $\big\{ \, \p( Y(v) \in \,\cdot\, \,|\, Y(w), w\in I\setminus\{v\} ) \, \big\}$ is a full conditional distribution of $Y$.\\
Suppose that $G$ is a graph whose nodes are partitioned into a conclique cover $C_1,\ldots,C_n$. Let $Y=(Y(v): v\in V)$ be a Markov random field on $G$ which takes values in $(S,\fS)$ with a full conditional distribution $\Big\{ F_v\left( Y(v) \in A \,|\, Y(w), w\in \operatorname{Ne}(v) \right) \: : v \in V \Big\}$ and an initial distribution $\mu_0$. Note that the joint conditional distribution of a conclique $Y(C_i)$ given its neighbors which are contained in $Y(C_1),\ldots,Y(C_{i-1}),Y(C_{i+1}),\ldots, Y(C_n)$ factorizes as the product of the single conditional distributions due to the Markov property. This means that we can simulate the stationary distribution of the Markov random field with a Markov chain (under mild regularity conditions).

\begin{algorithm}[Simulation of Markov random fields with concliques, \cite{kaiser2012}]\label{concliqueAlgo}
Simulate the starting values according to an initial distribution $\mu_0$ and obtain the vector of $Y^{(0)} = \left( Y^{(0)}(C_1), \ldots, Y^{(0)}(C_n) \right) $. In the next step, given a vector $Y^{(k)} = \left( Y^{(k)}(C_1), \ldots, Y^{(k)}(C_n) \right) $, simulate for $i=1,\ldots,n$ the concliques $Y^{(k+1)} (C_i)$ given the $(k+1)$-st simulation of the neighbors in $Y^{(k+1)}( C_1), \ldots, Y^{(k+1)}( C_{i-1})$ and $k$-th simulation of the neighbors in $Y^{(k)} (C_{i+1}),\ldots, Y^{(k)} (C_{n})$ with the specified full conditional distribution. Repeat this step, until the maximum iteration number is reached.
\end{algorithm}

The following example treats the multivariate normal distribution on a graph and can be found in \cite{cressie1993statistics}. Let $G=(V,E)$ be a finite graph and $\{ Y(v): \, v\in V \}$ be multivariate normal with expectation $\alpha \in \R^{|V|}$ and covariance $\Sigma \in \R^{ |V|\times |V|}$ in that $Y$ has the density
\[
		f_Y( y) = (2\pi)^{-\frac{d}{2}} \text{det}(\Sigma)^{-\frac{1}{2}} \exp \left\{	-\frac{1}{2} (y-\alpha)^T \Sigma^{-1} (y-\alpha) \right\}.
\]
Then for a node $v$ we have using the notation $P$ for the precision matrix $\Sigma^{-1}$
\[
		Y(v) \, |\, Y(-v) \sim \cN\left(  \alpha(v) - (P(v,v))^{-1}  \sum_{w \neq v} P(v,w) \Big( y(w) - \alpha(w) \Big)  , (P(v,v))^{-1} \right).
\]
Since $P = \Sigma^{-1}$ is symmetric and since we can assume that $\left(P(v,v) \right)^{-1} > 0$, $Y$ is a Markov random field if and only if for all nodes $v\in V$
\begin{align*}
		P(v,w) \neq 0 \text{ for all } w \in \operatorname{Ne}(v) \text{ and }
		P(v,w) = 0 \text{ for all } w \in V \setminus \operatorname{Ne}(v).
\end{align*}

\cite{cressie1993statistics} investigates the conditional specification
\begin{align}\label{concliquesMVN}
		Y(v) \,|\, Y(-v) \sim \cN\left( \alpha(v) + \sum_{w \in \operatorname{Ne}(v)} c(v,w) \big(Y(w) - \alpha(w) \big), \quad \tau^2(v) \right)
\end{align}
where $C=\big( c(v,w) \big)_{v,w}$ is a $|V|\times |V|$ matrix and $T = \text{diag}(\tau^2(v): v\in V)$ is a diagonal matrix such that the coefficients satisfy the necessary condition $\tau^2(v) c(w,v) = \tau^2(w) c(v,w)$ for $v\neq w$ and $c(v,v) = 0$ as well as $c(v,w) = 0 = c(w,v)$ if $v,w$ are no neighbors. This means $P(v,w) = -c(v,w) P(v,v)$, i.e., $\Sigma^{-1} = P =  T^{-1}  (I-C)$. If $I-C$ is invertible and $(I-C)^{-1} T$ is symmetric and positive definite, then the entire random field is multivariate normal with	$Y \sim \cN\left( \alpha, (I-C)^{-1} T \right)$.\\
With this insight it is possible to simulate a Gaussian Markov random field using concliques with a consistent full conditional distribution. In particular, it is plausible in many applications to use equal weights $c(v,w)$ (cf. \cite{cressie1993statistics}): we can write the matrix $C$ as $C=\eta H$ where $H$ is the adjacency matrix of $G$, i.e., $H(v,w)$ is 1 if $v,w$ are neighbors, otherwise it is 0. We know from the properties of the Neumann series that $I-C$ is invertible if $(h_0)^{-1} < \eta < (h_m)^{-1}$ where $h_0$ is the minimal and $h_m$ the maximal eigenvalue of $H$.

\subsection{A Numerical Example}
We can simulate a $d$-dimensional Markov random field $Z=(Z_1,Z_2,\ldots,Z_d)$ on a graph $G$ with the ansatz of \cite{cressie1993statistics}. The marginals of the single components $\{ Z_i(v): v\in V \}$ are standard normally distributed and the components $Z_1, \ldots, Z_d$ can be dependent among each other.\\
We give an example where we choose a lattice in two dimensions; the edge length is $40$ such that there are 1600 observations in total. We run on this lattice a Markov chain of $M_1=1000$ iterations for the simulation of a three-dimensional random  field $Z=(Z_1,Z_2,Z_3)$. Therefore we use a Gaussian copula to simulate $Z_1$ and $Z_2$ as dependent and $Z_3$ as independent. The correlation of $Z_1$ and $Z_2$ is approximately 0.7. \\
The parameter $\eta$ which describes the dependence within a random field $Z_i$ is chosen for all three components as $\eta = 0.25$. Note that $|\eta| \approx 0.25$ means a strong dependence whereas $\eta \approx 0$ indicates independence. In this case the admissible range for $\eta$ is very close to $(-0.257, 0.257)$. Note that the interval $(-0.25, 0.25)$ is the corresponding parameter space for a lattice wrapped on a torus. The influence of the parameter $\eta$ is as follows: if $\eta$ is positive and an observation $Z(v) > 0$, then $Z(v)$ increases the expectation of $Z(w)$ for all $w \in \operatorname{Ne}(v)$, see Equation~\eqref{concliquesMVN}. Conversely, if $\eta$ is negative, a positive observation $Z(v)$ decreases the expectation of its neighbors $Z(w)$.\\
In the next step, we construct a two-dimensional random field $( X_1,X_2)$ from $Z=(Z_1,Z_2,Z_3)$ and a one-dimensional random field with error terms $\epsilon$. For the error terms, we choose the independent component $Z_3$, thus, these are standard normally distributed. For $(X_1, X_2)$ we retransform $(Z_1,Z_2)$ as follows:\\
(a) We retransform each $Z_i$ with the inverse distribution function of the standard normal distribution to the interval [-1,1] and obtain $X_i$, hence, there remains a correlation between $X_1$ and $X_2$.\\
(b) We retransform $Z_i$ as in (a), additionally, we transform linearly all $X_2$ which are less than 0.1 onto $[0, 0.5]$ and the remaining $X_2$ onto $[0.5, 1]$, i.e.,
\begin{align}\label{ExampleSecondDistribution}
		X_2 \rightsquigarrow \frac{0.5}{0.1} \, X_2 \, \I\{X_2 < 0.1 \} + \left( \frac{ 0.5 - 0.1 }{1- 0.1}  + \frac{1- 0.5 }{1- 0.1} \, X_2\, \I\{ 0.1 \le X_2 \} \right).
\end{align}
Hence, in (a) the marginals of $(X_1,X_2)$ are approximately uniformly distributed on $[-1,1]$ and the correlation of $X_1$ and $X_2$ is approximately 0.68. In (b) the lower half of $[-1,1]^2$ contains approximately only 10\% of the data and the upper half 90\%. The correlation of $X_1$ and $X_2$ is approximately 0.65. The scatterplot of the two random fields is given in Figure~\ref{fig:Figure_RegressionFunction0}.\\
The regression functions are given as
\begin{align*}
	m_1(x) &\coloneqq 4 + 6x_1^2 - 4x_2^2 \quad\text{ and }\quad m_2(x) = m_1(x)\,\I\{ \norm{x}_2 \le 0.5 \} - m_1(x) \, \I\{\norm{x}_2 > 0.5 \}.
\end{align*}
All in all, we consider four different set-ups of the kind $Y(v) = m_i(X(v)) + \epsilon(s)$. For the estimation procedure, we choose our Haar basis from \eqref{HaarWaveletsLinSpace} and \eqref{HaarWaveletsLinSpace2}.\\
Now, let there be given a simulated random field $(X(v), Y(v): v\in V)$. We want to compare the estimator of the regression function $m$ which is obtained from this random field with the estimator which is obtained from an independent reference sample $\{(X_i,Y_i): i=1,\ldots,|V| \}$ of the same size. Therefore, we need to compute the $L^2$-error $\int_{\R^d} (\hat{m}-m)^2 \intd{\mu}$ in both cases. Usually, when computing the $L^2$-error, one partitions the data in a learning and in a test sample. The learning sample should comprise approximately 80\% of the data, cf. \cite{kohler2008nonlinear}. Then the estimator is computed from the learning sample and the $L^2$-error is computed from the test sample. However, since we want to compare the estimator from the dependent setting with the independent setting, we proceed in a different way. Namely, we compute for both estimators the $L^2$-error with a second independent sample of $X$ by Monte Carlo integration: let $\hat{m}_k$ be the estimator obtained from the random field resp. the independent sample for a certain threshold $\lambda$. Denote the second independent sample by $\{X'_i: i=1,\ldots,|V| \}$. Then the $L^2$-error is approximately,
$$L^2(\hat{m}_k) \approx |V|^{-1} \sum_{i=1}^{|V|} \left| \hat{m}_k( X'_i ) - m( X'_i ) \right|^2. $$
We point out the advantage of this method: if the law of the $X(s)$ cannot be given analytically, then numerical integration is unavoidable. However, in order to compare the estimators, there is the need for a neutral testing sample. This is the second independent reference sample $X'$.\\
This step is repeated $M_2=1000$ times and yields for a given threshold $\lambda$ an approximate mean and standard deviation of the $L^2$-error. Then we choose the threshold $\lambda$ which minimizes the $L^2$-error in the mean. The results are given in Table~\ref{TableExample}: the first table contains the results for the random field, the second those of the independent reference sample.\\
For the independent samples the design distribution of $X$ has in both cases correlations which match those of the respective dependent samples.\\ 
Note that in all cases the hard thresholding value $\lambda = 0.08$ yields the best fit. Furthermore, the $L^2$-error measure for independent samples is always better than for the corresponding dependent samples. The reason is the choice of $\eta$ which is here maximal and means a strong dependence within the lattice. We remark that further simulations show that a minor decrease in $\eta$ influences very positively the distribution of the $L^2$-error which then corresponds almost to that of the independent reference samples in all four cases.\\
Figures~\ref{fig:Figure_RegressionFunction1} and \ref{fig:Figure_RegressionFunction2} depict the best fit in each case for the dependent sample: one finds that the regression estimator is able to adapt both to the local smoothness of the underlying regression function and to the design distribution. This is quite pronounced for the second distribution where the fourth quadrant of $[-1,1]^2$ is only sparsely covered with data. In this area the estimator is in both cases nearly constant.

\begin{table}[ht]
\begin{center}
\begin{tabular}{|| c || c | c | c | c ||  }
\hline
\hline
	&	\multicolumn{4}{| c ||}{Estimates on two-dimensional lattice} \\
\hline
\hline
$\lambda$ & (a) with $m_1$ & (b) with $m_1$ & (a) with $m_2$ & (b) with $m_2$ \\
\hline
\hline
\multirow{2}{*}{0.0}	&	2.109	&	2.021	&	0.628	&	0.566	\\
	&	(0.502)	&	(0.328)	&	(0.176)	&	(0.112)	\\
\hline									
\multirow{2}{*}{0.04}	&	1.962	&	1.923	&	0.450	&	0.444	\\
	&	(0.485)	&	(0.324)	&	(0.153)	&	(0.106)	\\
\hline									
\multirow{2}{*}{0.08}	&	1.914	&	1.901	&	0.333	&	0.379	\\
	&	(0.469)	&	(0.320)	&	(0.108)	&	(0.087)	\\
\hline									
\multirow{2}{*}{0.12}	&	2.074	&	2.067	&	0.421	&	0.489	\\
	&	(0.482)	&	(0.343)	&	(0.110)	&	(0.091)	\\
\hline									
\multirow{2}{*}{0.16}	&	2.263	&	2.269	&	0.509	&	0.621	\\
	&	(0.499)	&	(0.364)	&	(0.125)	&	(0.096)	\\
\hline									
\multirow{2}{*}{0.20}	&	2.471	&	2.455	&	0.589	&	0.727	\\
	&	(0.489)	&	(0.361)	&	(0.139)	&	(0.103)	\\
\hline
\hline
	& \multicolumn{4}{| c ||}{Independent reference estimates} \\
\hline
\hline
$\lambda$ & (a) with $m_1$ & (b) with $m_1$ & (a) with $m_2$ & (b) with $m_2$ \\
\hline
\hline
\multirow{2}{*}{0.0}	&	1.998	&	1.900	&	0.564	&	0.501	\\
	&	(0.426)	&	(0.281)	&	(0.091)	&	(0.046)	\\
\hline									
\multirow{2}{*}{0.04}	&	1.855	&	1.800	&	0.388	&	0.376	\\
	&	(0.419)	&	(0.279)	&	(0.082)	&	(0.042)	\\
\hline									
\multirow{2}{*}{0.08}	&	1.788	&	1.758	&	0.253	&	0.285	\\
	&	(0.411)	&	(0.281)	&	(0.043)	&	(0.027)	\\
\hline									
\multirow{2}{*}{0.12}	&	1.944	&	1.904	&	0.332	&	0.385	\\
	&	(0.422)	&	(0.285)	&	(0.039)	&	(0.032)	\\
\hline									
\multirow{2}{*}{0.16}	&	2.133	&	2.111	&	0.417	&	0.520	\\
	&	(0.432)	&	(0.299)	&	(0.039)	&	(0.039)	\\
\hline									
\multirow{2}{*}{0.20}	&	2.329	&	2.294	&	0.483	&	0.632	\\
	&	(0.446)	&	(0.307)	&	(0.053)	&	(0.041)	\\
\hline
\hline
\end{tabular}
\caption{$L^2$-error of regression problems 1 - 4 based on 1000 simulations. For the dependent sample we run 1000 iterations of the MCMC algorithm of \cite{kaiser2012}. The estimated mean and in brackets the estimated standard deviation for a resolution $j=5$. The threshold $\lambda \approx 0.08$ is optimal in all cases. Note that the estimator from the dependent sample performs less well than the estimator from the independent reference sample both the mean and the standard deviation are bigger.}
\label{TableExample}
\end{center}
\end{table}

\begin{figure}[ht]\centering
	\includegraphics[width = \textwidth]{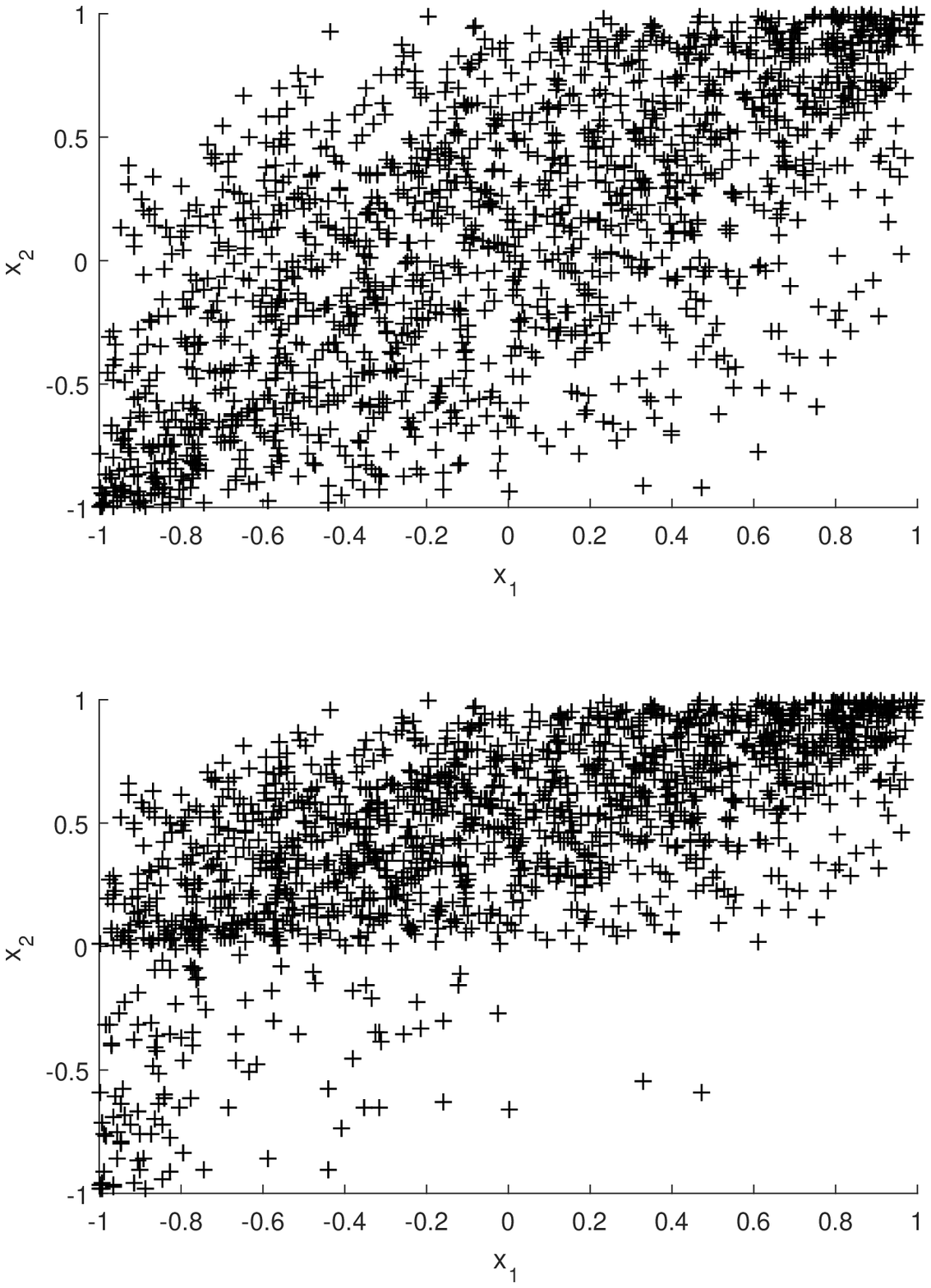}		
	\caption{Simulated scatterplot of the two 2-dimensional distributions. The first distribution (top) is symmetric. The second from Equation~\ref{ExampleSecondDistribution} (bottom) has little support in the fourth quadrant.}
	\label{fig:Figure_RegressionFunction0}
\end{figure}

\begin{figure}[ht]\centering
	\includegraphics[width = \textwidth]{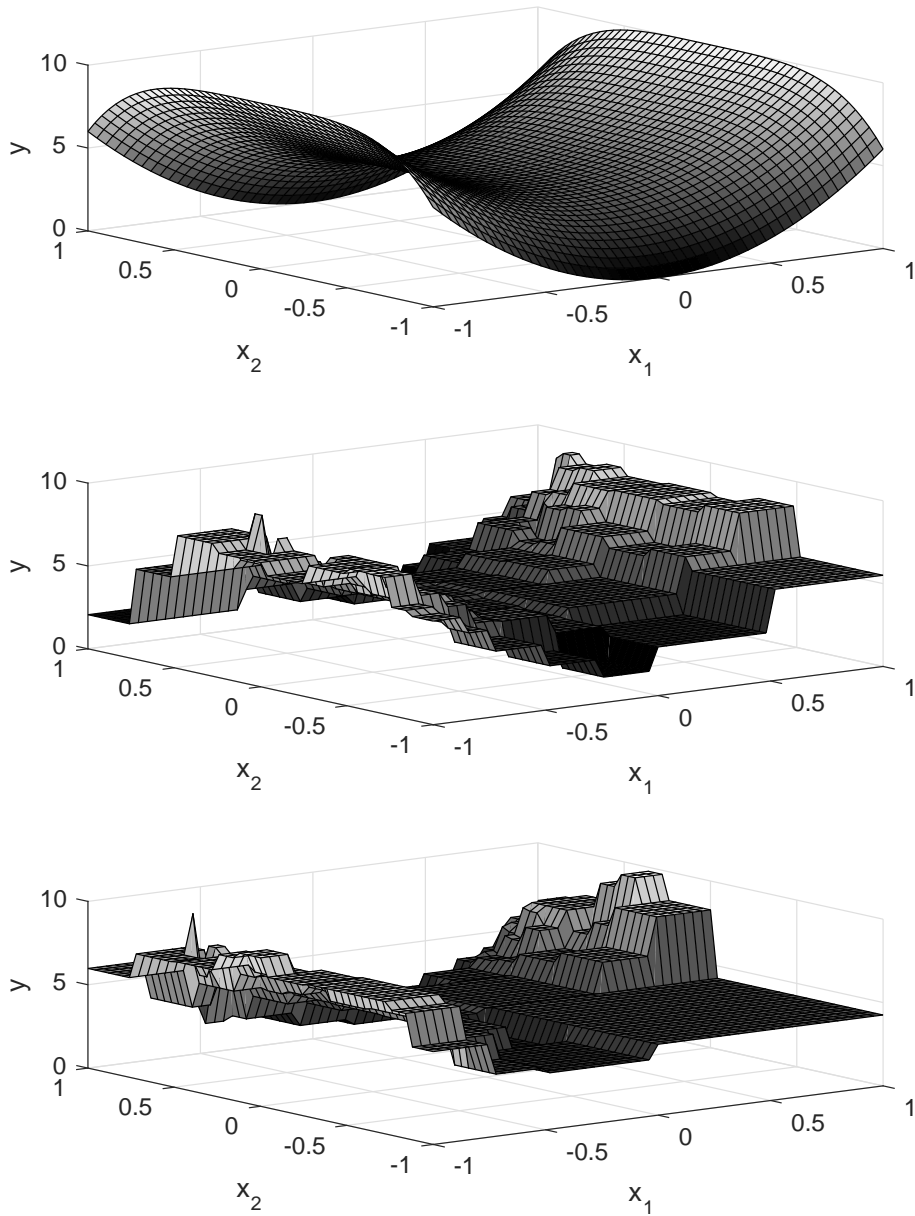}		
	\caption{True regression function $m_1$ (top) and estimates $\hat{m}_k$ with uniform data $X$ (middle) and nonuniform data (bottom). Note that the partition which is chosen depends on the data and on the local smoothness of the function. In particular, this is quite pronounced in the bottom figure in the fourth quadrant, i.e., $0\le X_1 \le 1$ and $-1\le X_2 \le 0$.}
	\label{fig:Figure_RegressionFunction1}
\end{figure}

\begin{figure}[ht]\centering	
		\includegraphics[width = \textwidth]{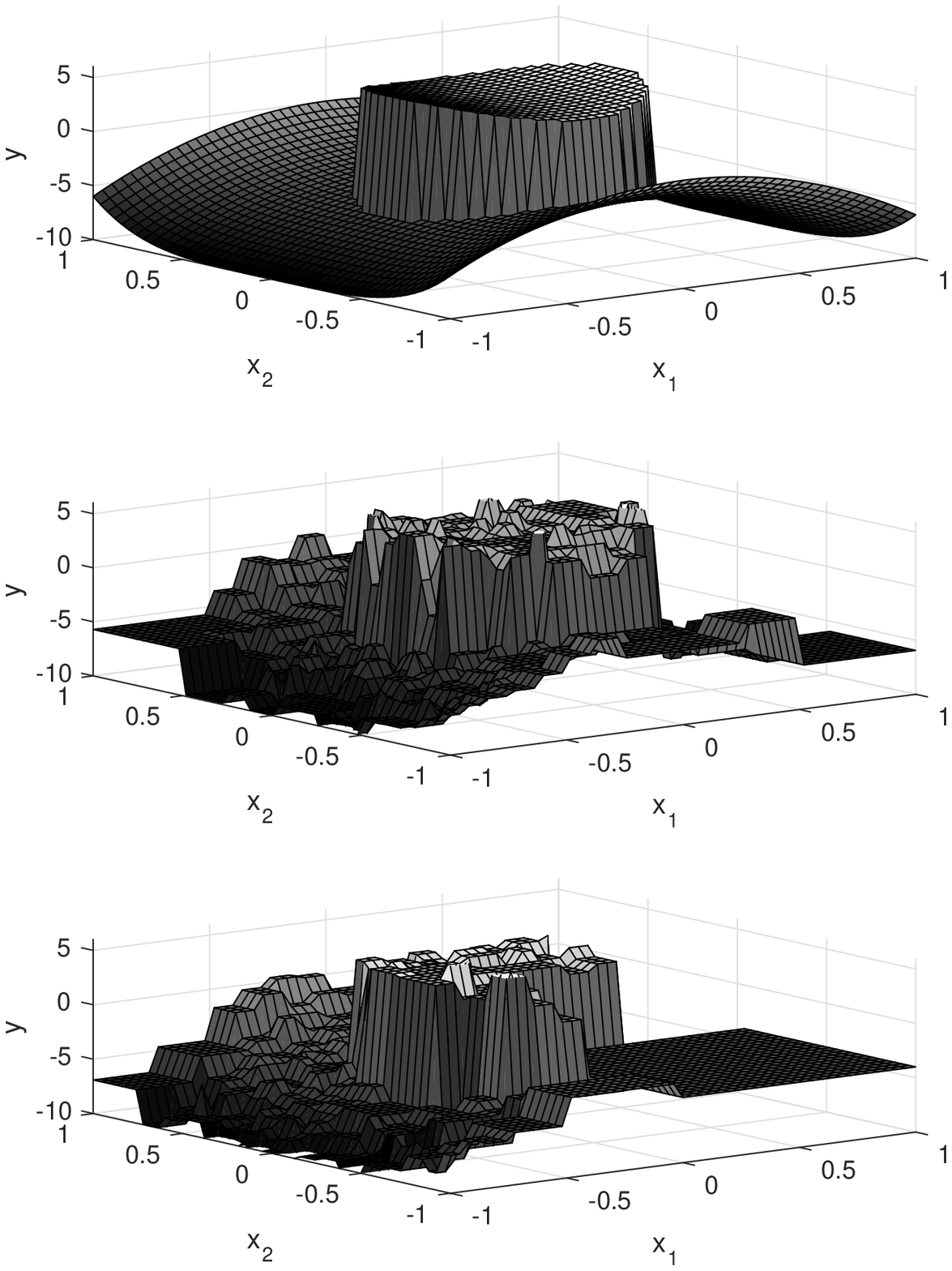}	
	\caption{True regression function $m_2$ (top) and estimates $\hat{m}_k$ with uniform data $X$ (middle) and nonuniform data (bottom). Again, the partition depends on the data and on the local smoothness of the function.}
	\label{fig:Figure_RegressionFunction2}
\end{figure}

\section{Proofs}\label{Section_Proofs}
We write $A$, $A_i$, $\tilde{A}_i$ resp. $C$, $C_i$ and $\tilde{C}_i$ for constants whose values are not necessarily the same.

\begin{proof}[Proof of Theorem~\ref{ConsisNonlinReg}]
We have with the defining property of $m$ and the properties of the conditional expectation for an independent observation $(X'(e_N),Y'(e_N))$
\begin{align*}
		&\E{ | \hat{m}_k(X'(e_N))-Y'(e_N) |^2 \,\big|\, X(I_{n(k)}),Y(I_{n(k)}) }  \\
		&\qquad=  \E{ | \hat{m}_k(X'(e_N))-m(X'(e_N)) |^2 \,\big|\, X(I_{n(k)}),Y(I_{n(k)}) } + \E{ | m(X'(e_N))-Y'(e_N) |^2 }.
		\end{align*}
Thus,
\begin{align*}
\int_{\R^d} | \hat{m}_k - m  | ^2 \intd{\mu} = \E{ |\hat{m}_k(X'(e_N))-Y'(e_N)|^2 \,\big|\, X(I_{n(k)}), Y(I_{n(k)}) } - \E{ |m(X'(e_N))-Y'(e_N) |^2 }.
\end{align*}
Since $\E{ |m(X'(e_N))-Y'(e_N) |^2 }$ is constant for all $k$ and $a - b = (\sqrt{a} - \sqrt{b} )(\sqrt{a} + \sqrt{b} )$, it suffices to prove that the following terms vanish for $k\rightarrow \infty$
\begin{align}
	0 &\le  \Bigg\{ \E{ (\hat{m}_k(X'(e_N))-Y'(e_N))^2 \,|\, X(I_{n(k)}), Y(I_{n(k)}) }^{1/2} - \inf_{\substack{f\in \cF_k, \\ \norm{f}_{\infty} \le \beta_k}} \E{ (f(X(e_N)) - Y(e_N) )^2 }^{1/2} \Bigg\} \nonumber \\
	& + \left\{ \inf_{\substack{f\in \cF_k, \\ \norm{f}_{\infty} \le \beta_k}} \E{ (f(X(e_N)) - Y(e_N) )^2 }^{1/2} - \E{ (m(X(e_N))- Y(e_N))^2}^{1/2} \right\} =: T_{1,k} + T_{2,k}. \label{NonlinWaveCons0}
\end{align}
The second term $T_{2,k}$ in \eqref{NonlinWaveCons0} converges to zero in the mean (resp. a.s.): this follows immediately with the reverse triangle inequality and the denseness assumption on the function spaces from \eqref{LinSpaceGeneral}; in the case of Haar wavelet spaces from \eqref{HaarWaveletsLinSpace} we need here that the sequence $(w_k: k\in\N)$ converges to infinity if the distribution of the $X(s)$ is not bounded in order to guarantee the denseness.\\
The first term $T_{1,k}$ in \eqref{NonlinWaveCons0} can be bounded in the following way (cf. again \cite{kohler2003nonlinear}) 
\begin{align*}
	T_{1,k} &\le 2 \E{ (Y(e_N)-Y_L(e_N) )^2 }^{1/2} + 2\left(	\frac{1}{|I_{n(k)}|} \sum_{s\in I_{n(k)}} (Y(s) - Y_L(s) )^2 \right)^{1/2} + \max_{ \widehat{J} \subseteq J} pen_k ( \widehat{J} ) \\
	&+ 2 \sup_{f \in T_{\beta_k} \cF_k} \left| \left(\frac{1}{|I_{n(k)}|} \sum_{s\in I_{n(k)}} |f(X(s))-Y_L(s) |^2\right)^{1/2} - \left( \E{ |f(X(e_N))-Y_L (e_N) |^2 } \right)^{1/2} \right|.
\end{align*}
Evidently, $pen_k( \widehat{J}) = \cO\left(	 \lambda_k^2\, K^*	\right) \rightarrow 0$ by assumption.
For $a.s.$ convergence of the entire term $T_{1,k}$, we need the ergodicity of the random field $\{ Y(s): s\in \Z^N \}$. This is guaranteed if the random field $Y$ is strong mixing and stationary by Theorem~\ref{MixingImpliesErgodicity}, Hence, we have $a.s.$
\begin{align*}
		&\left(	\frac{1}{|I_{n(k)}|} \sum_{s\in I_{n(k)}} (Y(s) - Y_L(s) )^2 \right)^{1/2} \rightarrow 	\E{ (Y(e_N)-Y_L(e_N) )^2 }^{1/2}  \text{ as } k \rightarrow \infty \\
		&\qquad\qquad \qquad \qquad\qquad \qquad  \qquad\qquad\text{ and } \E{ (Y(e_N)-Y_L(e_N) )^2 }^{1/2}  \rightarrow 0 \text{ as } L \rightarrow \infty.
\end{align*}
Consequently, it remains to show that
\begin{align}
		&S_k \coloneqq \sup_{f \in T_{\beta_k} \cF_k} \left| \left(\frac{1}{|I_{n(k)}|} \sum_{s\in I_{n(k)}} |f(X(s))-Y_L(s) |^2\right)^{1/2} - \left( \E{ |f(X(e_N))-Y_L(e_N) |^2 } \right)^{1/2} \right| \rightarrow 0, \label{NonlinWaveCons1}
\end{align} 
in the mean (resp. $a.s.$). For the convergence in the mean of \eqref{NonlinWaveCons1}, use the fact that $(\sqrt{a}-\sqrt{b})^2 \le |a-b|$, thus,
together with Hölder's inequality on probability spaces, the mean of $S_k$ satisfies
\[
	\E{ S_k } \le \E{  \sup_{f \in T_{\beta_k} \cF_k} \left| \left(\frac{1}{|I_{n(k)}|} \sum_{s\in I_{n(k)}} |f(X(s))-Y_L(s) |^2\right) - \left( \E{ |f(X(e_N))-Y_L(e_N) |^2 } \right) \right| }^{1/2},
\]
and apply Theorem~\ref{USLLNM} to the RHS. In case of $a.s.$-convergence, use again the relation $|\sqrt{a} - \sqrt{b} | \le \sqrt{|a-b|}$ and the continuity of the square root function. Hence, Theorem~\ref{USLLNM} applies in this case, too. In detail, we have for the tail distribution for $\epsilon > 0$ fix
\begin{align*}
		&\p\left( \sup_{f \in T_{\beta_k} \cF_k} \left| \left(\frac{1}{|I_{n(k)}|} \sum_{s\in I_{n(k)}} |f(X(s))-Y_L(s) |^2\right) - \left( \E{ |f(X(e_N))-Y_L(e_N) |^2 } \right) \right| > \epsilon \right) \\
		&\le \tilde{A}_1 H_{T_{\beta_k} \cF_k }\left( \frac{\epsilon}{128 \beta_k} \right) \exp\left\{ - \frac{ \tilde{A}_2 \epsilon \left( \prod_{i=1}^N n_i(k) \right)^{1/(N+1)}}{ \beta_k^2 \prod_{i=1}^N \log n_i(k) } 	\right\} \\
		&\le A_1 \exp\left\{ A_2 K^* \log \beta_k - \frac{ A_3  \left( \prod_{i=1}^N n_i(k) \right)^{1/(N+1)}}{ \beta_k^2 \prod_{i=1}^N \log n_i(k) }  \right\}, 
\end{align*}
where in the last inequality we use that the vector space dimension of $\cF_k$ is at most $K^*$. The constants $\tilde{A}_i, A_i$ depend on the lattice dimension, the bound on the mixing coefficients and $\epsilon>0$. One finds that \eqref{NonlinWaveCons1} converges to zero in the mean if
\[
			\beta_k^2\, K^*\, \log \beta_k \prod_{i=1}^{N} \log n_i(k) \Bigg/ \left(		\prod_{i=1}^N n_i(k) \right)^{1/(N+1)} \rightarrow 0 \text{ as } k \rightarrow \infty.
\]
$a.s.$-convergence of the term in \eqref{NonlinWaveCons1} follows with an application of the first Borel-Cantelli Lemma if additionally, for some positive $\delta > 0$
\[
		\beta_k^2 \, (\log k)^{1+\delta} \prod_{i=1}^{N}  \, \log n_i(k) \Bigg/ \left(		\prod_{i=1}^N n_i(k) \right)^{1/(N+1)} \rightarrow 0 \text{ as } k \rightarrow \infty.
\]
\end{proof}

\begin{lemma}[Variant of Lemma 1 in \cite{kohler2008nonlinear}]\label{RepresentationOfConstantFunctions}
Let $f\in \cF_c \circ \pi$ for a partition $\pi \in \prod_u$ for $1\le u \le u_{max} $. Then for fix $\omega\in\Omega$ there are balanced wavelets $f_{j_1},\ldots, f_{j_v} \in \cF_k$ which depend on this $\omega\in\Omega$, such that $\cF_c \circ \pi = \langle f_{j_1},\ldots,f_{j_v} \rangle$ in $L^2(\mu_n)$ and $v \le |\pi| = (2^d-1)(u-1) + (2 w_k)^d$.
\end{lemma}
\begin{proof}
The proof follows with induction on $1\le u \le u_{max} = 1 + (2w_k)^d (2^{(j_1-j_0)d}-1)/(2^d-1)$ and the definition of the set systems $\prod_u$. If $u=1$, then $\prod_1$ only contains the partition
\[
	\pi = \{ \supp{\Phi_{j_0,\gamma}}: \gamma\in A_{j_0,k} \} = \{ 2^{-j_0} [\gamma,\gamma+e_N): \gamma \in A_{j_0,k} \}.
	\]
For the inductive step, $u\rightarrow u+1$, let $\pi \in \prod_{u+1}$ be a partition and $\pi' \in \prod_{u}$ the corresponding predecessor partition which satisfies the relationship
\[
		\pi = \left(\pi' \setminus \{ A \} \right) \cup \{ A_1^{u_1} \times \ldots \times A_d^{u_d} : u_i \in \{L, R\}, i=1,\ldots,d \}.
\]
By construction, in $L^2(\mu_n)$ the following equality is true
\begin{align*}
		\cF_c \circ  \{ A_1^{u_1} \times \ldots \times A_d^{u_d} : u_i \in \{L, R\}, i=1,\ldots,d \} &= \left\langle\, \I\{ A_1^{u_1} \times \ldots \times A_d^{u_d}\} : u_i \in \{L,R\}, i=1,\ldots,d \,\right\rangle \\
		&= \left\langle\, \I\{A\}, f'_{1},\ldots,f'_{2^d - 1} \,\right\rangle,
\end{align*}
where the $f'_i$ are the orthonormal balanced wavelets on $A$ from our construction and $\I\{A\} \in \cF_c \circ \pi'$. By the inductive step, $\cF_c \circ \pi' = \left\langle\, f_{j_1},\ldots,f_{j_v}  \,\right\rangle$ for certain wavelets $f_{j_s}$ from the constructed orthonormal system with $v \le |\pi| = ( 2^d-1)(u-1) + (2 w_k)^d$. Hence, $\cF_c\circ \pi$ can be represented with $(2^d-1)u + (2w_k)^d$ elements as
\[
	\cF_c \circ \pi = \left\langle\, f_{j_1},\ldots,f_{j_v} , f'_1,\ldots, f'_{2^d-1} \,\right\rangle.
	\]
This finishes the proof.
\end{proof}

\begin{proof}[Proof of Theorem~\ref{NonlinRateOfConvExpDecr}] 
Let $(X'(I_{n(k)}), Y'(I_{n(k)}) )$ be an i.i.d.\ ghost sample on an enlarged probability space with the same marginal distributions as the observations $(X(s),Y(s))$ for the given sequence of index sets $I_{n(k)}$ from \eqref{Cond_Sequence} and \eqref{Cond_Sequence2}. Let the truncation sequence be given by $\beta_k \equiv B$. We define the empirical norms for a real valued function $f$ on $\R^d$
\[
	 \norm{f}_k \coloneqq \left(  |I_{n(k)} |^{-1} \sum_{s \in I_{n(k)} } f(X(s))^2 \right)^{1/2} \text{ and } \norm{f}'_k \coloneqq \left(  |I_{n(k)} |^{-1} \sum_{s \in I_{n(k)} } f(X'(s))^2 \right)^{1/2}. 
\]
Additionally, write $\norm{\,\cdot\,}$ for the $L^2(\mu)$-norm: $\norm{f}^2 = \int_{\R^d} f^2 \,\intd{\mu}$. The $L^2$-error decomposes in three terms
\begin{align*}
		\int_{\R^d} | \hat{m}_k-m |^2\,\intd{\mu}		&= \left\{ \norm{ \hat{m}_k - m}^2 - 2\left( (\norm{\hat{m}_k-m}'_k)^2 + pen_k(m_k) \right)		\right\} + 2\left\{	\norm{\hat{m}_k-m}_k^2 + pen_k(m_k)	\right\} \\
		&+ 2 \left\{	(\norm{ \hat{m}_k-m}'_k)^2 - \norm{ \hat{m}_k-m}_k^2	\right\} =: T_{1,k} + T_{2,k} + T_{3,k}.
\end{align*}
We investigate the terms $T_{i,k}$ separately. We start with $T_{1,k}$: note that $\hat{m}_k \in T_B \cF_k$, consequently,
\begin{align}
		&\p\left(	T_{1,k} > t	\right) = \p\left(	\norm{ \hat{m}_k-m}^2 - 2\left(	(\norm{ \hat{m}_k-m}'_k)^2 + pen_k(m_k)	\right)	> t \right) \nonumber \\
		&\le \p\left(	\sup_{f\in \cF_k}	\norm{ T_B f-m}^2 - 2\left( (\norm{T_B f-m}'_k)^2  + pen_k(f) \right) > t \right) \nonumber \\
		\begin{split}\label{NonlinUnb1} 
		&\le \p\Bigg(	\exists f\in T_{B} \cF_k: \E{ (f(X(e_N))-m(X(e_N)))^2 } - \frac{1}{|I_{n(k)}|} \sum_{s\in I_{n(k)}} \left( f(X'(s))-m(X'(s)) \right)^2	 \\
		&\qquad> \frac{1}{2} \left(		t + \E{ (f(X(e_N))-m(X(e_N)))^2 } \right) \Bigg), \end{split}
\end{align}
here we can omit the penalizing term because $pen_k(f) \ge 0$. Apply Lemma \ref{modGyorfiT11.4} to Equation \eqref{NonlinUnb1} with the parameters $\alpha = \beta = t/2$ and $\delta = 1/2$:
\begin{align}
	\eqref{NonlinUnb1} &\le 14\, H_{T_B\cF_k} \left(	\frac{ t/4}{20B} \right) \exp\left( - \frac{ t/2  \, |I_{n(k)}| }{2568\, B^4}		\right) \le C_1\, \exp\left( C_2\, K^* \,\log( B^2/t) - C_3\, t \, |I_{n(k)}| / B^4 \right), \nonumber
\end{align}
where we use that both $\log H_{T_B\cF_k} \left(	\frac{ t/4}{20B} \right) \le C_1 \cV_{(T_B \cF_k )^+ }\log \left(\frac{ C_2 B^2}{t} \right)$ and $ \cV_{(T_B \cF_k)^+} \le\cV_{\cF_k^+} \le K^*+1$. The constants $C_1,C_2, C_3$ do not depend on $I_{n(k)}$, $K^*$, $t$ or $B$. Hence, we choose $v_k \coloneqq K^*\, |I_{n(k)}|^{-1}\, \left(\prod_{i=1}^N \log n_i(k) \right)^2$ and the expectation of the first term can be bounded by
\begin{align}
	\E{ T_{1,k} } &\le v_k +  \int_{v_k}^{\infty} \p\left( T_{1,k} > t \right) \, \intd{t} = \cO\left( K^*  \, \left(\prod_{i=1}^N \log n_i(k) \right)^2 \Bigg/|I_{n(k)}|  \right). \label{NonlinUnb1b}
\end{align}
We study the second term $T_{2,k}$: therefore define the function which minimizes the penalized sum of squares
\[
		m_k^* \coloneqq \argmin{f \in \cF_k}  \left\{ \frac{1}{|I_{n(k)}|} \sum_{s\in I_{n(k)}} \left(	f(X(s))-m(X(s))	\right)^2	+ pen_k(f)	\right\}.
\]
We compute the conditional expectation of $T_{2,k}$ given the data $X(I_{n(k)})$ and use the pointwise inequality $|\hat{m}_k - m| \le |m_k-m|$ which is true because both $\hat{m}_k$ and $m$ are bounded by $B$,
\begin{align}
		&\frac{1}{2} \E{ T_{2,k} \,|\,X(I_{n(k)}) } \nonumber\\
		&\le \E{ \norm{m_k - m}_k^2 + pen_k(m_k) \,|\,X(I_{n(k)}) } \nonumber \\
		& \le \E{ \norm{m_k - m}_k^2 + pen_k(m_k) - 2\left( \norm{m^*_k - m}_k^2 + pen_k(m_k^*) \right) \,\Big|\,X(I_{n(k)}) } \nonumber \\
		&\qquad\qquad  + 2\left( \norm{m^*_k - m}_k^2 + pen_k(m_k^*) \right) \nonumber \\
		\begin{split}
		&\le v_k + \int_{v_k}^{\infty} \p\left(	\norm{m_k - m}_k^2 + pen_k(m_k) > 2\left( \norm{m^*_k-m}_k^2 + pen_k(m_k^*) \right) + t \,\Big|\, X(I_{n(k)})	\right) \, \intd{t} \\
		&\qquad\qquad + 2 \inf_{f\in\cF_k} \left\{ \frac{1}{|I_{n(k)}|} \sum_{s\in I_{n(k)}} \left(	f(X(s))-m(X(s))	\right)^2	+ pen_k(f)	\right\}, \label{NonlinUnb2} \end{split}
\end{align}
for $v_k > 0$ and where we use the continuity properties of a conditional distribution function as well as the defining property of $m_k^*$. Set $V^2(m_k|m_k^*) \coloneqq \norm{m_k-m_k^*}_k^2 + pen_k(m_k)$ and consider the conditional distribution in Equation \eqref{NonlinUnb2}: one can show with elementary calculations, cf. the proof of \cite{vandeGeer2001} Theorem 2.1 that by the definitions of $m_k$ and $m_k^*$ for given data $X(I_{n(k)}) = x(I_{n(k)}) \coloneqq \{ x_s: s\in I_{n(k)}\} \subseteq \R^d$ the inclusion
\begin{align}
	& \left\{ \norm{m_k - m}_k^2 + pen_k(m_k) > 2\left( \norm{m^*_k-m}_k^2 + pen_k(m_k^*) \right) + t \right\} \nonumber \\
	&\qquad\qquad \subseteq \left\{	\frac{1}{|I_{n(k)}|} \sum_{s\in I_{n(k)}} \varsigma(X(s))\, \epsilon(s) \, \Big( \,m_k(x(s))-m_k^*(x(s))\, \Big)\ge V^2(m_k|m_k^*)/12 \text{ and }	V^2(m_k|m_k^*) \ge t \right\} \nonumber
\end{align}
is true. Hence, the conditional distribution from Equation \eqref{NonlinUnb2} can be bounded as
\begin{align}
		&\p\left(\norm{m_k - m}_k^2 + pen_k(m_k) > 2\left( \norm{m^*_k-m}_k^2 + pen_k(m_k^*) \right) + t 	\right)\Bigg|_{X(I_{n(k)}) = x(I_{n(k)})}\nonumber  \\
		&\le \left. \p \left( \frac{1}{|I_{n(k)}|} \sum_{s\in I_{n(k)}} \varsigma(x(s))\, \epsilon(s) \; \Big( \,m_k(x(s))-m_k^*(x(s))\, \Big) \ge V^2(m_k|m_k^*)/12 \text{ and }	V^2(m_k|m_k^*) \ge t \right) \right|_{X(I_{n(k)}) = x(I_{n(k)})} \nonumber \\
		&\le \sum_{l=0}^{\infty}  \p\Bigg(	\frac{1}{|I_{n(k)}|} \sum_{s\in I_{n(k)}} \varsigma(x(s))\, \epsilon(s) \; \Big( \,m_k(x(s))-m_k^*(x(s))\, \Big) \ge \frac{2^{2l} t}{12}	 \nonumber \\
		& \qquad\qquad\qquad\qquad\qquad\qquad\qquad\qquad\qquad  \text{ and }	V^2(m_k|m_k^*) \in \left[2^{2l}t,2^{2(l+1)} t \right)  \Bigg) \Bigg|_{X(I_{n(k)}) = x(I_{n(k)})} \nonumber \\
		\begin{split}\label{NonlinUnb3}
		&\le \sum_{l=0}^{\infty}\p\Bigg(	\exists f\in\cF_k:  \frac{1}{|I_{n(k)}|} \sum_{s\in I_{n(k)}} \varsigma(x(s))\, \epsilon(s) \; \Big(\, f (x(s))-m_k^*(x(s)) \, \Big) \ge  \frac{2^{2l} t}{12}	  \\ 
		& \qquad\qquad\qquad \qquad\qquad\qquad\qquad\qquad\qquad\qquad  \text{ and }	V^2(f|m_k^*) \le 2^{2(l+1)} t \Bigg) \Bigg|_{X(I_{n(k)}) = x(I_{n(k)})}. 
		\end{split}
\end{align}
Thus, it suffices to show that \eqref{NonlinUnb3} can be bounded suitably. Define for $\delta > 0$ the functions classes
\[
	\cG_{k,l}(\delta)\coloneqq T_{\sqrt{ |I_{n(k)} | }\,  2^{l+1} \sqrt{\delta} } \; \left\{ f-m_k^*: f\in \cF_k, \; V^2(f|m_k^*) \le 2^{2(l+1)} \delta \right\}.
\]
The function class $\cG_{k,l}(t)$ corresponds to the functions used in \eqref{NonlinUnb3}. Note that we can truncate the functions at $\pm \sqrt{ |I_{n(k)} | }\,  2^{l+1} \sqrt{t}$ because the admissible $f\in\cF_k$ fulfill $|f(X(s))- m^*_k(X(s))| \le \sqrt{ |I_{n(k)} | }\,  2^{l+1} \sqrt{t}$. Set
\[
		R(n) \coloneqq |I_{n}|^{1/(N+1)} \Bigg/ \prod_{i=1}^N \log n_i \text{ for } n\in \N_+^N.
\]
Now we are able to apply Lemma \ref{BoundMixingNoise} to the probabilities in the sum in \eqref{NonlinUnb3} (with $K\coloneqq2^l \sqrt{t}$) 
\begin{align}
		\eqref{NonlinUnb3} &= \sum_{l=0}^{\infty} \p\left(	\sup_{g\in \cG_{k,l}  (t)} \left\{ \frac{1}{|I_{n(k)}|} \sum_{s\in I_{n(k)}} \varsigma(x(s))\, \epsilon(s) g(x_s) >  \frac{2^{2l} t}{12}	 \right\}  \right) \label{NonlinUnd4} \\
		\begin{split}\label{NonlinUnb5}
		&\le \sum_{l=0}^{\infty}  \inf_{D_1 > 0  } H_{\cG_{k,l}(t)} \left(\frac{ (2^{2l}\,t/ 12)^2}{ 8\cdot 2^{2l} t\; |I_{n(k)}|^{1/2} \; 2^{l+1} \sqrt{t} }\right) \\
		&\qquad \qquad \cdot \left\{	C_1 D_1^{1-\tau} (2^{2l}\, t)^{-1} \exp\left(- C_2 D_1^{\tau} \right)	+ C_1 \exp\left(- C_2\frac{2^{2l}t}{ D_1}  \,R(n(k)) \right) \right\}  \\
		&\qquad + \sum_{l=0}^{\infty}  \inf_{D_2 > 0} \left\{ C_1  D_2^{1-\tau/2}\, (2^{2l}\,t)^{-1}  \exp\left(- C_2 D_2^{\tau/2} \right)	+ C_1 \exp\left(- C_2\frac{2^{2l}t}{ D_2} R(n(k)) \right)  \right\},
		\end{split}
\end{align}
the constants $C_1$ and $C_2$ only depend on the lattice dimension $N$, the bound on the mixing coefficients and the tail parameters $\kappa_0,\kappa_1,\tau$. The covering number of this function classes $\cG_{k,l}(\delta)$ can be bounded with help of the Vapnik-Chervonenkis dimension of $\cF_k$
\[
		H_{\cG_{k,l}(\delta)}  \left(\frac{ (2^{2l}\, \delta / 12)^2}{ 8\cdot 2^{2l} \delta \, |I_{n(k)}|^{1/2} 2^{l+1} \sqrt{\delta } }\right) = H_{\cG_{k,l}(\delta)} \left( \frac{ 2^l \sqrt{\delta } } {2304 |I_{n(k)}|^{1/2} } \right) \le  (C\, |I_{n(k)}| )^{2 \cV_{\cG_k^+(\delta)}} \le (C \,  |I_{n(k)} |)^{ 2 (K^*+2)}
\]
because $\cG_{k,l}(\delta) \subseteq T_{ \sqrt{ |I_{n(k)}|} 2^{l+1} \sqrt{\delta} } \left\langle \cF_k,m_k^* \right\rangle$ and the $\R$-linear space  $\left\langle \cF_k,m_k^* \right\rangle$ has a vector space dimension of at most $K^*+1$; the bound can then be deduced from Proposition~\ref{boundCoveringNumber}.\\
Note that Equation \eqref{NonlinUnb5} is summable over the index $l$ for all $D_1,D_2 \in \R_+$ which are independent of $l$. We have again for suitable constants (which only depend on the lattice dimension $N$, the bound on the mixing coefficients and the tail parameters)
\begin{align}\begin{split}\label{NonlinUnb6}
		\eqref{NonlinUnb5} &\le (C_1 \, |I_{n(k)}| )^{ 2 (K^*+2)} \;\cdot\; \inf_{D_1 > 0  } \left\{ D_1^{1-\tau} t^{-1} \exp\left(- C_2 D_1^{\tau} \right)	+  \exp\left(- C_2 \frac{ t \,R(n(k)) } {D_1} \right) \right\}  \\
		&\qquad + C_1\;\cdot\; \inf_{D_2 > 0} \left\{   D_2^{1-\tau/2}\, t^{-1} \exp\left(- C_2 D_2^{\tau/2} \right)	+\exp\left(- C_2 \frac{ t \, R(n(k)) } {D_2} \right)  \right\}. 
		\end{split}\end{align}
Set the parameter $D_i$ for each $t$ such that the asymptotic growth rate of the two exponential terms are equal inside each factor of curly brackets of \eqref{NonlinUnb6}, i.e.
\[
		D_1 \coloneqq t^{1/(1+\tau)}\, \left(	|I_{n(k)}|^{1/(N+1)} \Big/ \prod_{i=1}^N \log n_i(k) \right)^{1/(\tau+1)} \text{ and } D_2 \coloneqq D_1^{(1+\tau)/(1+\tau/2) }.
\]
In particular, we find
\begin{align}
		\int_{v_k}^{\infty} \exp \left( - C_2 \, t^{\tau/(1+\tau)} \, R(n(k))^{\tau/(1+\tau)} \right) \,\intd{t} \le  C\,  \frac{(1+\tau)}{\tau} \frac{ v_k^{1/(1+\tau)} \exp\left( - C_2 \left(v_k\,R(n(k)) \right)^{\tau/(1+\tau)}  \right) }{ R(n(k))^{\tau/(1+\tau)} } . \label{NonlinUnb7}
\end{align}
In addition, we have $D_1^{1-\tau} t^{-1} = t^{-2\tau/(1+\tau)}\, R(n(k))^{(1-\tau)/(1+\tau)}$, hence, this factor is decreasing in $t$. Define
\[
		v_k \coloneqq \left( K^* \left( \prod_{i=1}^N \log n_i(k) \right)^2 \right)^{(1+\tau)/\tau} \Bigg/ R(n(k)).
\]
If we combine \eqref{NonlinUnb6} with \eqref{NonlinUnb7}, we find that the integral from $v_k$ to $\infty$ over the integrand in the first line of \eqref{NonlinUnb6} decreases at a speed which is asymptotically in
\[
	\cO\left( (K^* (\prod_{i=1}^{N} \log n_i(k))^2 )^{(1+\tau)/\tau} / R(n(k)) \right).
	\]
In the same way, by formally replacing $\tau$ with $\tau/2$, one finds that the integral over the integrand in the second line in \eqref{NonlinUnb6} is in $\cO\left( (K^* (\prod_{i=1}^{N} \log n_i(k))^2 )^{(2+\tau)/\tau} / R(n(k)) \right)$. With this reduction, we can estimate the integral in Equation \eqref{NonlinUnb2} as
\begin{align*}
		 &\int_{v_k}^{\infty} \p\left(	\norm{m_k - m}_k^2 + pen(m_k) > 2\left( \norm{m^*_k-m}_k^2 + pen(m_k^*) \right) + t \,\Big|\, X(I_{n(k)})	\right) \, \intd{t}  \\
		&\le C \frac{ (K^* (\prod_{i=1}^{N} \log n_i(k))^2 )^{(2+\tau)/\tau} }{R(n(k)) },
\end{align*}
where the constant $C$ only depends on the lattice dimension $N$, the bound on the mixing coefficients and the tail parameters. Hence, the expectation of $T_{2,k}$ is bounded by
\begin{align}
	&\E{ T_{2,k} } \nonumber\\
	&\le 4 \E{ \inf_{f\in\cF_k} \left\{ \frac{1}{|I_{n(k)}|} \sum_{s\in I_{n(k)}} \left|	f(X(s))-m(X(s))	\right|^2	+ pen_k(f)	\right\} } + C\, \frac{ (K^* (\prod_{i=1}^{N} \log n_i(k))^2 )^{(2+\tau)/\tau} }{R(n(k)) } \nonumber \\
	&\le 4 \inf_{f \in \cF_k } \left\{ \int_{\R^d} |f-m|^2\, \intd{\mu} + \E{ pen_k(f) } \right\} + C\, \frac{ (K^* (\prod_{i=1}^{N} \log n_i(k))^2 )^{(2+\tau)/\tau} }{R(n(k)) } \label{NonlinUnb10}\\
	&\le 4 \min_{1\le u \le K^*} \left\{	\inf_{\substack{ f\in \cF_k,\\ f = \sum_{i=1}^u a_i g_i }} \int_{\R^d} |f-m|^2\, \intd{\mu} + u \lambda_k^2	\right\} + C\, \frac{ (K^* (\prod_{i=1}^{N} \log n_i(k))^2 )^{(2+\tau)/\tau} }{R(n(k)) }. \nonumber
\end{align}
Especially in the case of the wavelet system we can bound \eqref{NonlinUnb10} slightly better if we use Lemma~\ref{RepresentationOfConstantFunctions}:
\begin{align}
		\eqref{NonlinUnb10} \le 4 \min_{1\le u \le u_{max} } & \left\{	\lambda_k^2 \left( (2^d-1)(u-1) + (2w_k)^d \right) 	+ \min_{\pi \in \prod_u} \inf_{f \in \cF_c \circ \pi } \int_{\R^d} |f-m|^2\,\intd{\mu} \right\} \nonumber \\
		&\quad + C\, \left(  (2\cdot 2^{j_1-j_0} w_k)^d \left(\prod_{i=1}^{N} \log n_i(k) \right)^2 \right)^{(2+\tau)/\tau} \Bigg/ R(n(k)), \nonumber
\end{align}
where $u_{max} = 1+ (2w_k)^d [ (2^{d(j_1-j_0)}-1)/(2^d-1) ]$ is the maximum index of the sets of partitions given in Equation~\eqref{uMax}. We consider the third term. Define the function class	$\cG_k \coloneqq \left\{g_f \coloneqq (f-m )^2 : f\in T_B \cF_k		\right\}$. Let $f_1,\ldots,f_v$ be an $\tilde{\epsilon}$-cover of $T_B \cF_k$ w.r.t.\ the $L^1$-norm of the empirical measure of the points $(x_1,\ldots,x_u) \subseteq \R^d$. As both $m$ and the functions $f$ in $T_B \cF_k$ are bounded by $B$, we have that the functions in $\cG_k$ are bounded by $4B^2$. Furthermore, the functions $g_{f_i}(x) \coloneqq (f_i(x) -m(x) )^2 $  ($i=1,\ldots,v$) are a $4B\tilde{\epsilon}$-cover of $\cG_k$ w.r.t.\ the $L^1$-norm of the empirical measure induced by $x_1,\ldots,x_u \subseteq \R^{d}$. Indeed, let $f\in T_B \cF_k$ be in the neighborhood of $f_j$ and denote by $g_f$ resp. $g_{f_j}$ the corresponding functions, then 
\begin{align}
   \frac{1}{u} \sum_{i=1}^u  \left|\, g_f(x_i) - g_{f_j}(x_i) \,\right| &= \frac{1}{u} \sum_{i=1}^u \left|\, (f(x_i)-m(x_i)) ^2 - (f_j(x_i) -  m(x_i) )^2 \,\right| \nonumber \\
	&\le \frac{4B}{u} \sum_{i=1}^u \left|\, f(x_i)-f_j(x_i) \,\right| \le 4B\,\tilde{\epsilon}.\nonumber
\end{align}
Consequently, $H_{\cG_k}(t/4) \le H_{T_B \cF_k}(t/(16 B))$ and with Lemma \ref{ConvDistDepAndInd}, we obtain for the distribution of $T_{3,k}$ the following inequalities
\begin{align}
		&\p\left( 	(\norm{ \hat{m}_k-m}'_k)^2 - \norm{ \hat{m}_k-m}_k^2 > t	\right) \nonumber \\
		&\le \p\left(	\sup_{f\in T_B \cF_k} \left| \frac{1}{|I_{n(k)}|} \sum_{s\in I_{n(k)}} (f(X'(s))-m(X'(s)))^2 - \frac{1}{|I_{n(k)}|} \sum_{s\in I_{n(k)}} (f(X(s))-m(X(s)))^2	\right| > t \right) \nonumber \\
		&\le \cH_{ T_B \cF_k }\left( \frac{t}{16 B }	\right)\, \sup_{j} \p\Biggl(	\Biggl| \frac{1}{|I_{n(k)}|} \sum_{s\in I_{n(k)}} (f_j(X'(s))-m(X'(s)))^2 \nonumber \\
		& \qquad\qquad \qquad\qquad \qquad\qquad - \frac{1}{|I_{n(k)}|} \sum_{s\in I_{n(k)}} (f_j(X(s))-m(X(s)))^2	\Biggl| > \frac{t}{2}	\Biggl) \nonumber \\
		&\le C_1 \,\exp\left( C_2 K^* \,\log(B^2/t) - C_3 R(n(k))\, t / B^2 \right),
\end{align}
for suitable constants $C_1, C_2, C_3\in\R_+$ which only depend on the lattice dimension $N$, the bound on the mixing coefficients. Here, we use $\cV_{T_B \cF_k^+} \le K^* + 1$. Hence, the expectation of the first term is bounded as
\begin{align*}
	\frac{1}{2}	\E{ T_{3,k} } &\le v_k + C_1 \, \exp\left( C_2\, K^* \log (1/v_k) \right)\; \int_{v_k}^{\infty} \exp\left( -C_3 t R(n(k)) \right)\,\intd{t} \\
	&= \cO\left( K^*\, \left(\prod_{i=1}^N \log n_i(k) \right)^2 \Bigg/ R(n(k)) \right).
\end{align*}
All in all, $T_{1,k}$ and $T_{3,k}$ are both negligible and the asymptotic properties are determined by $T_{2,k}$.
\end{proof}

\begin{proof}[Proof of Theorem~\ref{RateConvNonlinThrm}]
The proof can be carried out in the same way as the proof of Theorem \ref{NonlinRateOfConvExpDecr}. The bounds on the terms $T_{1,k}$ and $T_{3,k}$ do not change, both terms are in $\cO\left( K^*\, \left(\prod_{i=1}^N \log n_i(k) \right)^2 \big/ R(n(k)) \right)$. The second term can be treated in the same way until Equation \eqref{NonlinUnd4}. Here use Theorem \ref{USLLNM} to obtain constants
\begin{align*}
		\eqref{NonlinUnb4} &\le \sum_{l=0}^{\infty} C_1 \left( \frac{ C_2 \sqrt{|I_{n(k)}| } \sqrt{t} 2^{l+1} }{2^{2l} t } 	\right)^{2( K^* +2)} \, \exp\left( - C_3 2^{2l} \,t \,R(n(k))	\right) \\
		&\le C_1 \left(\frac{ C_2 \sqrt{ |I_{n(k)}| }}{\sqrt{t}} \right)^{2(K^*+2)} \exp\left(- C_3 t R(n(k)) \right).
\end{align*}
With this bound it is straightforward to show
\begin{align*}
	\E{ T_{2,k} } &\le  4 \inf_{f \in \cF_k } \left\{ \int_{\R^d} |f-m|^2\, \intd{\mu} + \E{ pen_k(f) } \right\}  + C\, \frac{ K^* \left( \prod_{i=1}^N \log n_i(k) \right)^2 }{ R(n(k)) }
\end{align*}
and we are back in Equation \eqref{NonlinUnb10}. In this case, the constant $C$ only depends on the lattice dimension $N$ and the bound on the mixing coefficients. This finishes the proof.
\end{proof}

\section*{Supplementary Material}
A supplement \cite{krebsOrthogonalSuppl} gives further technical results for the simulation procedure which is used in Section~\ref{Section_SimulationExamples}.

\section*{Acknowledgments}
The author is indebted to an Associate Editor and a referee for thoughtful suggestions and comments which significantly clarified and improved the manuscript.

\appendix

\section{Exponential inequalities for dependent sums}\label{Appendix_ExpInequalities}

We start with a definition of the covering number:

\begin{definition}[$\epsilon$-covering number]
Let $\left( \R^d,\cB(\R^d) \right)$ be endowed with a probability measure $\nu$ and let $\cG$ be a set of real valued Borel functions on $\R^d$ and let $\epsilon>0$. Every finite collection $g_1,\ldots,g_N$ of Borel functions on $\R^d$ is called an $\epsilon$-cover of $\cG$ w.r.t. $\norm{\,\cdot\,}_{L^p(\nu)}$ of size $N$ if for each $g\in\cG$ there is a $j$, $1\le j\le N$, such that $\norm{g-g_j}_{L^p(\nu)} < \epsilon$. The $\epsilon$-covering number of $\cG$ w.r.t. $\norm{\,\cdot\,}_{L^p(\nu)}$ is defined as
\[
		\mathsf{N}\left( \epsilon, \cG, \norm{\,\cdot\,}_{L^p(\nu)} \right) := \inf\left\{ N\in \N: \exists\, \epsilon-\text{cover of }\cG\text{ w.r.t. }\norm{\,\cdot\,}_{L^p(\nu)} \text{ of size }N \right\}.
\]
Evidently, the covering number is monotone: $\mathsf{N}\left(\epsilon_2, \cG, \norm{\,\cdot\,}_{L^p(\nu)} \right) \le \mathsf{N}\left(\epsilon_1, \cG, \norm{\,\cdot\,}_{L^p(\nu)}\right)$ if $\epsilon_1 \le \epsilon_2$.
\end{definition}

The covering number can be bounded uniformly over all probability measures for a class of bounded functions under mild regularity conditions. Thus, the following covering condition is appropriate for many function classes $\cG$.

\begin{condition}[Covering condition]\label{coveringCondition}
$\cG$ is a class of uniformly bounded, measurable functions $f:\R^d \rightarrow \R$ such that $\norm{f}_{\infty} \le B < \infty$ and for all $\epsilon > 0$ and all $N\ge 1$ the following is true:
\par
\begingroup
\leftskip=1cm
\rightskip=1cm 
\noindent
For any choice $z_1,\ldots,z_M \in \R^d$ the $\epsilon$-covering number of $\cG$ w.r.t.\ the $L^1$-norm of the discrete measure with point masses $\frac{1}{M}$ in $z_1,\ldots,z_M$ is bounded by a deterministic function depending only on $\epsilon$ and $\cG$, which we shall denote by $H_{\cG} (\epsilon)$, i.e., $\mathsf{N}\left(\epsilon,\cG, \norm{\,\cdot\,}_{L^1(\nu)} \right) \le H_{\cG} (\epsilon)$., where $\nu = \frac{1}{M} \sum_{k=1}^M \delta_{z_k}$.
\par
\endgroup
\end{condition}

Denote by $\cG^+ := \Big\{ \big\{ (z,t)\in \R^d\times\R: t\le g(z) \big\} : g\in \cG \Big\}$ the class of all subgraphs of the class $\cG$. Condition \ref{coveringCondition} is satisfied if the Vapnik-Chervonenkis dimension of $\cG^+$ is at least two, i.e., $\cV_{\cG^+}\ge 2$ and if $\epsilon$ sufficiently small:

\begin{proposition}[Bound on the covering number, \cite{haussler1992decision}]\label{boundCoveringNumber}
Let $[a,b] \subset \R$ be a finite interval. Let $\cG$ be a class of uniformly bounded real valued functions $g: \R^d \mapsto [a,b]$ such that $\cV_{\cG^+} \ge 2$. Let $0 < \epsilon < (b-a)/4$. Then for any probability measure $\nu$ on $\cB(\R^d)$
\begin{align*}
		\mathsf{N}\left( \epsilon, \cG, \norm{\,\cdot\,}_{L^p(\nu)} \right) \le 3 \left( \frac{2e (b-a)^p}{\epsilon^p} \,\log\frac{3e(b-a)^p}{\epsilon^p} \right)^{\cV_{\cG^+}}.
\end{align*}
In particular, in the case that $\cG$ is an $r$-dimensional linear space, we have $\cV_{\cG^+} \le r+1$.
\end{proposition}

The Bernstein inequality from \cite{frankeBernstein} from Theorem~\ref{frankeBernstein} puts us in position to formulate the inequality which yields upper bounds on probability of the event of the type
\begin{align}
		\left \{ \sup_{g\in\cG} \left|	\frac{1}{|I_n|} \sum_{s \in I_n} g(Z(s)) - \E{g(Z(e_N) )} \right| > \epsilon \right  \}. \label{supMeas}
\end{align}
Of course, \eqref{supMeas} is not an event for general function classes, however, we assume that the function classes in the present context are sufficiently regular such that \eqref{supMeas} is $\cA$-measurable.

\begin{theorem}[Bernstein inequality for spatial lattice processes]\label{frankeBernstein}
Let $Z:=\{ Z(s): s \in \Z^N \}$ be a real-valued random field defined on $\Z^N$. Let $Z$ be strong mixing with mixing coefficients $\{\alpha(k): k\in \N_+\}$ such that each $Z(s)$ is bounded by a uniform constant $B$ and has expectation zero and the variance of $Z(s)$ is uniformly bounded by $\sigma^2$. Furthermore, put $\bar{\alpha}_k := \sum_{u=1}^k u^{N-1} \alpha(u)$. Let $P(n), Q(n)$ be non-decreasing sequences in $\N_+^N$ which are indexed by $n\in\N_+^N$ and which satisfy for each $1\le i \le N$
\begin{align*}
		&1 \le Q_i(n_i) \le P_i(n_i) < Q_i (n_i) + P_i(n_i) < n_i.
		\end{align*}
Furthermore, let $\tilde{n} := |I_n| = n_1\cdot \ldots \cdot n_N$, $\tilde{P} := P_1(n_1) \cdot \ldots \cdot P_N (n_N) $ and $\underline{q} := \min\left\{ Q_1(n_1),\ldots,Q_N(n_N) \right\}$ as well as $\overline{p} := \max\left\{ P_1(n_1), \ldots, P_N(n_N) \right\}$. Then for all $\epsilon > 0$ and $\beta >0$ such that $2^{N+1} B \tilde{P} e \beta < 1$
\begin{align}\begin{split}\label{frankeBernsteinEq0}
		\p\left( \left|\sum_{s \in I_n} Z(s) \right| > \epsilon \right) &\le 2 \exp\left\{ 12 \sqrt{e} 2^N \frac{\tilde{n}}{\tilde{P}} \alpha(\underline{q})^{ \tilde{P} \big/ \left[ \tilde{n} \left(2^N+1\right) \right] }	\right\} \\
		&\qquad\qquad\qquad \cdot  \exp\left\{	-\beta \epsilon + 2^{3N} \beta^2 e \left(	\sigma^2 + 12 B^2 \gamma \bar{\alpha}_{\overline{p}} \right) \tilde{n}	\right\},
\end{split}\end{align}
where $\gamma$ is a constant which depends on the lattice dimension $N$.
\end{theorem}
\begin{proof}
A proof can be found in \cite{frankeBernstein}.
\end{proof}

We can formulate the following extension of the above Bernstein inequality
\begin{theorem}\label{extBernstein}
Let $\{Z(s): s\in I \}$ be a strong mixing random field with $\E{Z(s)} = 0$ and $\E{Z(s)^2} \le \sigma^2 < \infty$. Furthermore, assume that the tail distribution is bounded by
\begin{align}
		\p(|Z(s)| > z ) \le \kappa_0 \exp\left( - \kappa_1 z^{\tau} \right) \label{tailBound}
\end{align}
for $\kappa_0,\kappa_1, \tau >0$. Then, for any $B>0$, we have with the notation from Theorem \ref{frankeBernstein}
\begin{align*}
		\p\left( \left| \sum_{s\in I_n} Z(s) \right| > \epsilon \right) &\le \frac{12}{\epsilon \tau} \kappa_0  \kappa_1^{-\frac{1}{\tau}} \Gamma\left( \tau^{-1}, \kappa_1 B^{\tau} \right) |I_n| + 2 \exp\left\{ 12 \sqrt{e}2^N \frac{\tilde{n}}{\tilde{P}} \alpha(\underline{q})^{ \tilde{P}\big/\left[\tilde{n} \left(2^N+1\right) \right] } \right\} \\
		&\qquad\qquad\qquad\qquad \cdot \exp\left\{ - \frac{1}{3} \beta \epsilon \right\} \cdot \exp\left\{	2^{3N}\beta^2 e\left( \sigma^2 + 48 B^2 \gamma\, \bar{\alpha}_{\overline{p}} \right) \tilde{n}	\right\}
\end{align*}
where $\Gamma$ denotes the upper incomplete $\Gamma$ function.
\end{theorem}
\begin{proof}
A proof can be found in \cite{frankeBernstein}.
\end{proof}

We give two results which are immediate consequences of Theorems \ref{frankeBernstein} and \ref{extBernstein}: 
\begin{proposition}\label{applBernstein}
Let the real valued random field $Z$ satisfy Condition \ref{regCond0} \ref{Cond_Distribution} and \ref{Cond_Mixing}. The $Z(s)$ have expectation zero and are bounded by $B$. Let $n\in \N_+^N$ be such that both
$$\min_{1\le i \le N} n_i \ge e^2 \text{ and } \frac{ \min\{ n_i: i=1,\ldots,N \}}{ \max\{ n_i: i=1,\ldots,N \} } \ge C',$$
for a constant $C'> 0$. There are constants $A_1, A_2\in \R_+$ which depend on the lattice dimension $N$, the constant $C'$ and the bound on the mixing coefficients but not on $n\in\N_+^N$ and not on $B$ such that for all $\epsilon> 0$ 
\begin{align*}
		\p\left( \left| \sum_{s\in I_n} Z(s) \right| > \epsilon \right) \le A_1 \,\exp\left( - A_2 \epsilon \,B^{-1}\, \left(\prod_{i=1}^N n_i \right)^{-N/(N+1)} \, \left( \prod_{i=1}^N \log n_i \right)^{-1} \right).
\end{align*}
\end{proposition}
\begin{proof}
A proof can be found in \cite{frankeBernstein}.
\end{proof}

\begin{theorem}[A uniform concentration inequality]\label{USLLNM}
Let $Z$ be a random field on $\pspace$ which satisfies Condition \ref{regCond0} \ref{Cond_Distribution} and \ref{Cond_Mixing}. Let $\cG$ be a set of measurable functions $g: \R^d \rightarrow [0,B]$ for $B\in [1,\infty)$ which satisfies Condition \ref{coveringCondition}.  Let $n\in \N_+^N$ be such that both
$$\min_{1\le i \le N} n_i \ge e^2 \text{ and } \frac{ \min\{ n_i: i=1,\ldots,N \}}{ \max\{ n_i: i=1,\ldots,N \} } \ge C',$$
for a constant $C'> 0$. Then given that \eqref{supMeas} is measurable$[\cA\,|\,\cB(\R^d)]$, for any $\epsilon > 0$
\begin{align*}
&\p\left( \sup_{g \in \cG} \left| \frac{1}{|I_n| } \sum_{ s \in I_n} g(Z(s)) - \E{ g(Z(e_N) )}		\right| > \epsilon		\right) \\
&\qquad\qquad \le A_1 \, H_{\cG} \left(\frac{\epsilon}{32} \right) \left\{ \exp\left( - \frac{A_2 \, \epsilon^2 \, |I_n| }{B^2}	\right) + \exp\left( - \frac{A_3\, \epsilon\, |I_n| }{B\, \left( \prod_{i=1}^N n_i \right)^{N/(N+1) }\, \prod_{i=1}^N \log n_i }	\right)		\right\}
\end{align*}
where the constants $A_1,A_2$ and $A_3$ only depend on the lattice dimension $N$, $C'$ and on the bound on the mixing coefficients given by $c_0, c_1 \in \R$ in Condition~\ref{regCond0} \ref{Cond_Mixing}.
\end{theorem}

Since in practice, we shall use the bound given in Theorem \ref{USLLNM} on an increasing sequence $(n(k): k \in \N) \subseteq \N_+^N$ and on increasing function classes $\cG_k$ whose essential bounds $B_k$ increase with the size of the index sets $I_{n(k)}$, it is possible to omit the first factor in the above theorem under certain conditions: let a sequence of function classes $\cG_k$ with bounds $B_k$ and a sequence $(\epsilon_k: k\in\N_+)\subseteq \R_+$ be given such that
\[
	\lim_{k \rightarrow \infty} \epsilon_k |I_{n(k)}| \, \Bigg/ \left\{  B_k\, \left(\prod_{i=1}^N n_i(k) \right)^{N/(N+1) }  \prod_{i=1}^N \log n_i (k)  \right\} = \infty,
	\]
then the above equation reduces to
\begin{align*}
&\p\left( \sup_{g \in \cG_k } \left| \frac{1}{|I_{n(k)} | } \sum_{ s \in I_{n(k)} } g(Z(s)) -  \E{ g(Z(e_N) )}		\right| > \epsilon_k		\right) \\
&\qquad\qquad\qquad\qquad\qquad\qquad\qquad\le A_1 \, H_{\cG_{k} } \left(\frac{\epsilon_k}{32} \right) \, \exp\left( - \frac{A_2\, \epsilon_k |I_{n(k)} | } {B_k \left( \prod_{i=1}^N n_i (k) \right)^{ N/(N+1) }  \prod_{i=1}^N \log n_i(k) }	\right)
\end{align*}
with new constants $A_1,A_2\in\R_+$.

\begin{proof}[Proof of Theorem \ref{USLLNM}]
We assume the probability space to be endowed with the i.i.d.\ random variables $Z'(s)$ for $s \in I_n$ which have the same marginal laws as the $Z(s)$. We write for shorthand
\[
		S_n(g) := \frac{1}{ |I_n|} \sum_{s\in I_n} g(Z(s)) \text{ and } S'_n(g) := \frac{1}{|I_n|} \sum_{s\in I_n} g(Z'(s)).
\]
Thus, we can decompose
\begin{align}
	&\p\left(	\sup_{g\in\cG} \left| S_n(g) - \E{ g(Z(e_N) )} \right| > \epsilon	\right) \nonumber\\
		&\le \p\left(	\sup_{g \in \cG} \left| S_n(g) - S'_n(g) \right| > \frac{\epsilon}{2}\right) + \p\left(	\sup_{g \in \cG} \left| S'_n(g) - \E{ g(Z'(e_N) )} \right| > \frac{\epsilon}{2}\right) \label{USLLNM1}
\end{align}
and apply Theorem 9.1 from \cite{gyorfi} to second term on the right-hand side of \eqref{USLLNM1} which is bounded by
\begin{align}
		\p\left(	\sup_{g \in \cG} \left| S'_n(g) - \E{ g(Z'(e_N) )} \right| > \frac{\epsilon}{2}\right) \le 8 H_{\cG} \left( \frac{\epsilon}{16} \right) \exp\left( - \frac{|I_n| \epsilon^2}{512 B^2} \right). \label{USLLNM1b} 
\end{align}
To get a bound on the first term of the right-hand side of \eqref{USLLNM1}, we apply for fix $\omega \in \Omega$ Condition \ref{coveringCondition} to the set $\{ Z(s,\omega), Z'(s,\omega) : s\in I_n \}$. Let $g_k^{\ast}(\omega)$ for $k=1,\ldots,H^*:=H_{\cG} \left( \frac{\epsilon}{32} \right)$ be chosen as in Condition \ref{coveringCondition}, possibly with some redundant $g^*_k(\omega)$ for $\tilde{H}(\omega) < k \le H^*$ where $\tilde{H}(\omega)$ is the number of non-redundant functions. Note that $H^*$ is deterministic. Define the random sets for $k=1,\ldots,H^*$ by
\[
		U_k(\omega) := \left\{ g\in \cG: \frac{1}{2|I_n|} \sum_{s\in I_n} \Big|g(Z(s,\omega))-g^{\ast}_k(Z(s,\omega))\Big| + \Big|g(Z'(s,\omega)) - g^{\ast}_k(Z'(s,\omega))\Big| < \frac{\epsilon}{32}		\right\},
\]
note that some $U_k(\omega)$ might be redundant for $\tilde{H}(\omega) < k \le H^*$. This implies that for each $\omega\in \Omega$ we can write $\cG = U_1(\omega) \cup \ldots \cup U_k(\omega)$, consequently,
\begin{align}
		&\p\left(	\sup_{g \in \cG} \left| S_n(g) - S'_n(g) \right| > \frac{\epsilon}{2}\right) = \p\left(	\max_{1\le k \le H^*} \sup_{g \in U_k} \left| S_n(g) - S'_n(g) \right| > \frac{\epsilon}{2}\right) \nonumber \\
		&\qquad \qquad \qquad \le \E{ \sum_{k=1}^{\tilde{H}} 1_{ \left\{ \sup_{g \in U_k} |S_n(g) - S'_n(g)| > \frac{\epsilon}{2} \right\} } } \le \sum_{k=1}^{H^*} \p\left(	\sup_{g \in U_k} \left| S_n(g) - S'_n(g) \right| > \frac{\epsilon}{2}\right). \label{USLLNM2}
\end{align}
In the following we suppress the $\omega$-wise notation; let now $g \in U_k$ be arbitrary but fix, then
\begin{align}
		| S_n(g) - S'_n(g) | \le 2 \frac{\epsilon}{32} + | S_n(g^*_k) - S'_n(g^*_k) |. \label{USLLNM3}
\end{align}
Thus, using Equation \eqref{USLLNM3}, we get for each summand in \eqref{USLLNM2}
\begin{align}
	&\p\left(	\sup_{g \in U_k} \left| S_n(g) - S'_n(g) \right| > \frac{\epsilon}{2}\right) \le \p\left( \left| S_n(g^*_k) - S'_n(g^*_k) \right| > \frac{7\epsilon}{16} 	\right) \nonumber \\
	&  \le \p\left( \left|S_n(g^*_k) - \E{ g^*_k(Z(e_N) )} \right|	>\frac{7\epsilon}{32} 	\right)  + 
	 \p\left( \left|S'_n(g^*_k) - \E{ g^*_k(Z'(e_N) )} \right|	>\frac{7\epsilon}{32} 	\right). \label{USLLNM4}
\end{align}
The second term on the right-hand side of \eqref{USLLNM4} can be estimated using Hoeffding's inequality, we have
\begin{align}
		\p\left( \left|S'_n(g^*_k) - \E{ g^*_k(Z'(e_N))} \right|	>\frac{7\epsilon}{32} 	\right) \le 2 \exp \left\{ - \frac{98 \, |I_n| \, \epsilon^2}{32^2 \, B^2}		\right\}. \label{USLLNM5}
\end{align}
We apply the Bernstein inequality for strong spatial mixing data from Theorem \ref{frankeBernstein} to the first term of Equation \eqref{USLLNM4}. We obtain for the first term on the right-hand side of \eqref{USLLNM4} with Proposition \ref{applBernstein}
\begin{align}
		\p\left( \left|S_n(g^*_k) - \E{ g^*_k(Z(e_N))} \right|	>\frac{7\epsilon}{32} 	\right) &\le 2 A_1 \exp\left(	-   \frac{ A_2 \epsilon|I_n| }{ B \left( \prod_{i=1}^N n_i \right)^{N/(N+1)} \prod_{i=1}^N \log n_i } 	\right). \label{USLLNM6}
\end{align}
And all in all, using that $H_{\cG} \left(\frac{\epsilon}{16} \right) \le H_{\cG} \left(\frac{\epsilon}{32} \right)$ and with the help of Equation  \eqref{USLLNM1b},  and Equations \eqref{USLLNM5} and \eqref{USLLNM6} plugged in \eqref{USLLNM4} and that again in \eqref{USLLNM2} we get the result - using the notation $\tilde{n} = \prod_{i=1}^{N} n_i $
\begin{align*}
		&\p\left(	\sup_{g\in\cG} \left| \frac{1}{ |I_n| } \sum_{s\in I_n} g(Z(s) ) - \E{ g(Z(e_N) )} \right| > \epsilon	\right) \\
		&\le 	8 H_{\cG} \left( \frac{\epsilon}{16} \right) \exp\left( - \frac{  \epsilon^2\, |I_n| }{512 B^2} \right)
		+ 2 H_{\cG} \left( \frac{\epsilon}{32} \right) \left\{ \exp\left( - \frac{98   \epsilon^2\, |I_n| }{32^2 B^2}		\right) + 	A_1 \exp\left(	- \frac{A_2  \epsilon \, |I_n| }{B\,  \tilde{n} ^{N/(N+1)  }\, \prod_{i=1}^N \log n_i}	\right) \right\} \\
		& \le \big(10 + 2 A_1 \big) \, H_{\cG} \, \left(\frac{\epsilon}{32}\right) \left\{ \exp\left( -\frac{ \epsilon^2}{512} \frac { |I_n| } { B^2} \right) + \exp\left( - \frac{A_2  \epsilon \, |I_n|  }{B\,\tilde{n} ^{N/(N+1)}\, \prod_{i=1}^N \log n_i}	\right)\right\}.
\end{align*}
This finishes the proof.
\end{proof}

It follow the lemmata which we need for the proof of Theorem~\ref{NonlinRateOfConvExpDecr}:
\begin{lemma}[Large deviations of strong mixing samples from independent samples]\label{ConvDistDepAndInd}
Let the random field $Z$ satisfy Condition \ref{regCond0} \ref{Cond_Distribution} and \ref{Cond_Mixing}. Furthermore, let $Z'$ be an i.i.d.\ ghost sample with the same marginals as $Z$. Let $\cG$ be a class of functions $g:\R^d\rightarrow \R$ which are uniformly bounded by $B\in\R_+$ and fulfill Condition ~\ref{coveringCondition}. Let $n\in \N_+^N$ be such that both
$$\min_{1\le i \le N} n_i \ge e^2 \text{ and } \frac{ \min\{ n_i: i=1,\ldots,N \}}{ \max\{ n_i: i=1,\ldots,N \} } \ge C',$$
for a constant $C'> 0$. Then, there are constants $0<A_1,A_2<\infty$ which only depend on $N$, $C'$ and the bound on the mixing coefficients such that for all $\epsilon > 0$
\begin{align}
		\p\left(	\sup_{g\in\cG} \left| \frac{1}{|I_n|} \sum_{s\in I_n} g(Z(s)) - \frac{1}{|I_n|} \sum_{s\in I_n} g(Z'(s))  \right| > \epsilon \right) &\le A_1 \, H_{\cG}\left(	\frac{\epsilon}{4} \right)\, \exp\left( -\frac{   A_2 \epsilon\, \left( \prod_{i=1}^N n_i \right)^{1/(N+1)} }{B \prod_{i=1}^N \log n_i }		\right). \nonumber
\end{align}
\end{lemma}
\begin{proof}[Proof of Lemma~\ref{ConvDistDepAndInd}]
Let $g_1,\ldots,g_{N^*}$ be an $\epsilon/4$-covering of $\cG$ with respect to the $L^1$-norm of the empirical measure induced by $(Z(I_n), Z'(I_n)) \subseteq \R^d$. Then
\begin{align}
		&\p\left(	\sup_{g\in\cG} \left| \frac{1}{|I_n|} \sum_{s\in I_n} g(Z(s)) - \frac{1}{|I_n|} \sum_{s\in I_n} g(Z'(s))  \right| > \epsilon \right) \nonumber\\
		&\qquad\qquad\qquad\qquad\qquad\qquad\le H_{\cG}\left( \frac{\epsilon}{4}	\right) \sup_{1\le j \le N^*} \p\left( \left| \frac{1}{|I_n|} \sum_{s\in I_n} g_j(Z(s)) - \frac{1}{|I_n|} \sum_{s\in I_n} g_j(Z'(s))  \right| > \frac{\epsilon}{2} \right). \nonumber
\end{align}
The claim follows now with an application of Proposition~\ref{applBernstein}.
\end{proof}

\begin{lemma}[Modified version of Theorem 11.4 of \cite{gyorfi}]\label{modGyorfiT11.4}
Let $(X(i),Y(i): i=1,\ldots,n)$ be an independent sample for the regression problem from Equations \eqref{lsqI}. Assume that the regression function $m$ is essentially bounded, $\norm{m}_{\infty} \le B < \infty$, for $B\ge 1$. Let $\cF$ be a function class where each element fulfills $f:\R^d\rightarrow\R$ and $\norm{f}_{\infty} \le B$. Then given that $\alpha,\beta,\gamma > 0$ and $0 < \delta \le 1/2$
\begin{align*}
	&\p\Bigg(	\sup_{f\in\cF} \E{ |f(X(e_N))-m(X(e_N)) |^2}  - \frac{1}{|I_n|} \sum_{s\in I_n} \left\{ |f(X(s))-m(X(s)) |^2 \right\}  \nonumber \\	
	&\qquad\qquad \ge \delta \left(\alpha+\beta + \E{ |f(X(e_N))-m(X(e_N)) |^2} \right) \Bigg) \le 14 H_{\cF} \left( \frac{\beta\,\delta}{20B} \right) \,\exp\left\{- \frac{ \delta^2(1-\delta)\,\alpha\, |I_n| }{214\,(1+\delta)\,B^4}	\right\}.
\end{align*}
\end{lemma}
\begin{proof}
One can deduce the claim from the proof of Theorem 11.4 of \cite{gyorfi}.
\end{proof}

\begin{lemma}[Large deviations for heteroscedastic noise]\label{BoundMixingNoise}
Let the random field $\epsilon = \{\epsilon(s): s\in \Z^N\}$ fulfill Condition \ref{regCond0} \ref{Cond_Distribution} and \ref{Cond_Mixing}, have zero means and satisfy the tail condition	
\[
	\p( |\epsilon(s)| > z ) \le \kappa_0 \exp( -\kappa_1 z^{\tau} ) \text{  for constants } 0 < \kappa_0,\kappa_1,\tau < \infty.
	\]
Let the function class $\cG$ fulfill Condition \ref{coveringCondition} for functions $g:\R^d\rightarrow \R$ and $\norm{g}_{\infty} \le B$; $B \ge 1$. Let $\varsigma: \R^d \rightarrow \R_+$ be essentially bounded. Let $n\in \N_+^N$ be such that both
$$\min_{1\le i \le N} n_i \ge e^2 \text{ and } \frac{ \min\{ n_i: i=1,\ldots,N \}}{ \max\{ n_i: i=1,\ldots,N \} } \ge C',$$
for a constant $C'> 0$.  Let $\{x_s:s\in I_n\}$ be points in $\R^d$ where $I_n = \{s: e_N \le s\le n \}$. Furthermore, let  $K\in\R_+$. Then for two constants $A_1, A_2 \in \R_+$ which depend on $N$, $C'$, the bound on the mixing coefficients and the tail parameters $\kappa_0,\kappa_1,\tau$ but which are independent of $B$, $K$, $\delta$ and $n$
\begin{align*}
		&\p\left(	\sup_{g\in\cG}  \left| \frac{1}{|I_n|} \sum_{s\in I_n} \varsigma(x_s) \epsilon(s) \, g(x_s) \right| > \delta \right) \\
		& \le \inf_{D_1 > 0  } H_{\cG} \left(\frac{ \delta^2}{8K^2B}\right) \left\{	A_1 D_1^{1-\tau} \delta^{-1} \exp\left(- A_2 D_1^{\tau} \right)	+ A_1 \exp\left(- \frac{ A_2 \delta \left(\prod_{i=1}^N n_i \right)^{1/(N+1)} }{ D_1  \prod_{i=1}^N \log n_i } \right) \right\} \nonumber \\
		&  \qquad\qquad + \inf_{D_2 > 0} \left\{ A_1 \norm{\varsigma}_{\infty}^2 K^{-2} D_2^{1-\tau/2} \exp\left(- A_2 D_2^{\tau/2} \right)	+ A_1 \exp\left(- \frac{ A_2 \norm{\varsigma}_{\infty}^{-2} K^2 \, \left(\prod_{i=1}^N n_i \right)^{1/(N+1)} }{ D_2  \prod_{i=1}^N \log n_i  } \right) \right\}.
\end{align*}
\end{lemma}
\begin{proof}
We use the extended Bernstein inequality for unbounded random variables from Theorem \ref{extBernstein}: here we can bound $\Gamma\left(1/\tau, c_1 B^{\tau} \right)$ by $c_0 (c_1 B^{\tau})^{-1+1/\tau  } \exp( - c_1 B^{\tau} )$ for a suitable constant $c_0 \in \R_+$ which depends on $\tau$ but not on $B$ and on $c_1$. We apply Theorem~\ref{extBernstein} to a random field $W$ which has zero means and fulfills the tail condition $\p(|W(s)| > z ) \le \kappa_0 \exp( - \kappa_1 z^{\tau} )$: there are suitable constants $A_1, A_2 \in\R_+$ which only depend on $\kappa_0,\kappa_1,\tau$, the lattice dimension $N$ and the bound on the mixing coefficients but not on $n$ and $\delta$ such that
\begin{align*}
	\p\left( \left| \sum_{s\in I_n} W(s) \right| > \delta \, |I_n|	\right) \le \inf_{D > 0} \,A_1 D^{1-\tau} \delta^{-1} \exp\left( - A_2 D^{\tau} \right) + A_1 \exp\left( - \frac{ A_2 \delta \left(\prod_{i=1}^N n_i \right)^{1/(N+1)}} { D   \prod_{i=1}^N \log n_i  } \right).
\end{align*}
Furthermore, let there be given an $\tilde{\delta}$-covering of $\cG$ w.r.t. the $L^1$-norm induced by the empirical measure $|I_n|^{-1} \sum_{s\in I_n} \delta_{x_s}$ which we denote by $\{g_1,\ldots,g_{N^*} \}$, for some $N^*\in\N_+$. Then, any function $g$ in the $\tilde{\delta}$-neighborhood of a covering function $g_j$ satisfies
\[
		\sqrt{ |I_n|^{-1}  \sum_{s \in I_n} | g(x_s) - g_j (x_s) |^2 } \le \sqrt{ |I_n|^{-1}  \sum_{s \in I_n} | g(x_s) - g_j (x_s) |\, 2B } \le \sqrt{2B\,\tilde{\delta}}.
\]
I.e., $\{g_1,\ldots,g_{N^*} \}$ is a $\sqrt{2B\,\tilde{\delta}}$-covering w.r.t. the 2-norm. This means the $\delta$-covering number w.r.t. the 2-norm is bounded by $H_{\cG} \left( \delta^2 / 2B \right)$. Let now $K\in \R_+$ be given, then the desired probability is bounded by:
\begin{align}
	&\p\left(	\sup_{g\in\cG}  \left| \frac{1}{|I_n|} \sum_{s\in I_n} \varsigma(x_s) \, \epsilon(s) \, g(x_s) \right| > \delta \right)  \nonumber \\
	\begin{split}\label{NonlinUnb4}
	&\le \p\left(	\sup_{g\in\cG}  \left| \frac{1}{|I_n|} \sum_{s\in I_n} \varsigma(x_s) \, \epsilon(s) \, g(x_s) \right| > \delta \text{ and } \frac{ \norm{\varsigma}^2_{\infty} }{|I_n|} \sum_{s\in I_n} \epsilon(s)^2 \le K^2 \right) \\
	& \qquad\qquad\qquad \qquad\qquad\qquad \qquad\qquad\qquad + \p\left( 	\frac{ \norm{\varsigma}^2_{\infty} }{|I_n|} \sum_{s\in I_n} \epsilon(s)^2 > K^2 	\right).
	\end{split}
\end{align}
Let there be given a $(\delta/(2K))^2/(2B)$-covering of $\cG$ with respect to the $L^1$-norm of the measure $|I_n|^{-1} \sum_{s\in I_n} \delta_{x_s}$ which is an $\delta/(2K)$-covering w.r.t. the corresponding 2-norm. Observe that the random field $\epsilon^2= \{ \epsilon(s)^2: s\in \Z^N \}$ fulfills the tail condition  with $\tau/2$.\\
Furthermore, $\p\left(	|\varsigma(x_s) \, \epsilon(s) \, g(x_s) | > z \right) \le \p\left( |\epsilon(s)| > (\norm{\varsigma}_{\infty} B)^{-1} z \right)$, so for these random variables the constants in tail condition changes somewhat. Altogether, we can bound \eqref{NonlinUnb4} as follows: apply the $\delta/(2K)$-covering $\{ g_1, \ldots, g_{N^*} \}$ and the Cauchy-Schwarz inequality to the first term inside the first probability, then
\begin{align}
		\eqref{NonlinUnb4} &\le H_{\cG} \left( \frac{(\delta/2K)^2}{2B} \right) \, \sup_{1 \le j \le N^*} \p\left( \left| \frac{1}{|I_n|} \sum_{s\in I_n} \varsigma(x_s) \epsilon(s) \, g_j(x_s) \right| \ge \frac{\delta}{2}		\right) +  \p\left( 	\frac{ 1}{|I_n|} \sum_{s\in I_n} \epsilon(s)^2 > \frac{K^2}{\norm{\varsigma}^2_{\infty} } 	\right) \nonumber \\
		&\le \inf_{D_1 > 0  } H_{\cG} \left( \frac{\delta^2}{8K^2B} \right)  \left\{	A_1 D_1^{1-\tau} \delta ^{-1} \exp\left(- A_2 D_1^{\tau} \right)	+ A_1 \exp\left(-\frac{ A_2 \delta \left( \prod_{i=1}^N n_i \right)^{1/(N+1)} }{ D_1  \prod_{i=1}^N \log n_i } \right) \right\} \nonumber \\
		&\qquad + \inf_{D_2 > 0} \left\{ A_1 \norm{\varsigma}_{\infty}^2 K^{-2} D_2^{1-\tau/2} \exp\left(- A_2 D_2^{\tau/2} \right)	+ A_1 \exp\left(- \frac{ A_2 \norm{\varsigma}_{\infty}^{-2} K^2  \left( \prod_{i=1}^N n_i \right)^{1/(N+1)} }{D_2 \prod_{i=1}^N \log n_i } \right) \right\} \nonumber
\end{align}
where the constants $A_1,A_2$ are independent of $B$, $K$, $n$, $\delta$ and the $D_i$ but depend on the lattice dimension, the bound on the mixing coefficients and the tail parameters $\kappa_0,\kappa_1$ and $\tau$. This finishes the proof.
\end{proof}

\section{Ergodic theory for spatial processes}\label{AppendixB}
In the next lines, we give a review on important concepts of ergodicity when dealing with random fields on subgroups of the discrete group $\Z^N$. For further reading consult \cite{tempelman2010ergodic}.

\begin{definition}[Dynamical systems and ergodicity]\label{dynSys}
Let $\pspace$ be a probability space and $(G,+)$ a locally compact, abelian Hausdorff group which fulfills the second axiom of countability. We write for $x,y\in G$ arbitrary $x-y$ for $x+(-y)$ and $-y$ is the $+$-inverse of $y$. Furthermore, let $\nu$ be a Haar measure on $\cB(G)$, i.e. for all $x \in G$ and for all Borel sets $B\in\cB(G)$ we have $\nu(B) = \nu( x + B)$.\\
A family of bijective mappings $\{ T_x: \Omega \rightarrow \Omega,\, x\in G\}$ is called a flow if it fulfills the following three conditions
\begin{enumerate}
	\item $T_x$ is measure-preserving, i.e. $\p(A) = \p( T_x A)$  for all $A \in \cA$ and for all $x \in G$,
	\item $T_{x + x' } = T_x \circ T_{x'}$ and $T_{x}\circ T_{-x} = Id_{\Omega}$ for all $x,x' \in G$,
	\item	the map $G\times \Omega \ni (x,\omega) \mapsto T_x \omega $ is measurable[$\cB(G)\otimes\cA\,|\,\cA$].
\end{enumerate}
Let $T=\{ T_x: x\in G\}$ be a flow in $\pspace$, then the quadruple $(\Omega,\cA,\p, T)$ is called a \textit{dynamical system}. The dynamical system is called ergodic if the invariant $\sigma$-field $\cI := \{ A \in \cA: A = T_x A \, \forall x \in G\}$ is trivial[$\p$], i.e. if for all $A \in\cI$ we have $\p(A)\in \{0,1\}$.\\
Let now $\Gamma \le \Z^N$ be a subgroup and $Z = \left\{Z(s): s\in \Gamma \right\}$ be a stationary random field on $\pspace$ where each $Z(s)$ takes values in the measure space $(S,\fS)$. Let $\nu$ be the counting measure on $\cB(\Gamma)$. Set $\p_Z := \p_{	\{ Z(s): s\in \Gamma \}}$ for the probability measure on $\otimes_{s \in \Gamma } \fS$ induced by the finite dimensional distributions of $Z$ and define on the path space $\left(\times_{s\in \Gamma} S,\otimes_{s\in \Gamma}\fS, \p_Z \right)$ the family of translations
\[
	T_t: \times_{s\in \Gamma } S \rightarrow \times_{s\in \Gamma} S, \Big(z(s): s\in \Gamma \Big) \mapsto \Big(z(s+t) : s \in \Gamma \Big) \quad \text{ for } t \in \Gamma,
\]
which is a flow because $Z$ is stationary. Then $Z$ is called ergodic if and only if the quadruple $\left(\times_{s\in \Gamma} S,\otimes_{s\in \Gamma}\fS, \p_Z, T \right)$ is ergodic.
\end{definition}

The next result is an extension of Birkhoff's celebrated ergodic theorem it can be found in \cite{tempelman2010ergodic}
\begin{theorem}[Ergodic theorem, \cite{tempelman2010ergodic}]\label{ergodicTheorem}
Let $(\Omega,\cA,\p,T)$ be a dynamical system. Furthermore, let $\{ W_n: n\in \N\} \subseteq G$ be an increasing sequence of Borel sets of G such that $0 < \nu(W_n) < \infty$ for all $n\in\N$ which fulfills both
\[
	\lim_{n\rightarrow \infty} \frac{ \nu( W_n \cap (W_n -x) ) }{\nu(W_n)} = 1 \text{ for all } x \in G \text{ and } \sup_{n \ge 0} \frac{ \nu( W_n - W_n)}{\nu(W_n)} < \infty,
\]
where $W_n - W_n := \{ x-y: x,y\in W_n \}$. Then, for an integrable random variable $X\in L^1(\p)$ 
\[
		\lim_{n\rightarrow \infty} \frac{1}{ \nu(W_n)} \int_{W_n} X( T_x \omega) \nu(\intd{x}) = \E{ X\,|\, \cI}(\omega) \quad \text{ for $\p$-almost every} \, \omega\in\Omega.
\]
\end{theorem}
\begin{proof}
Confer \cite{tempelman2010ergodic} Chapter 6, in particular Proposition 1.3 and Corollary 3.2.
\end{proof}

We are now prepared to state a well-known and useful result, cf. \cite{hannan2009multiple} Theorem IV.2 and the discussion thereafter for a treatment of one-dimensional stochastic processes.
\begin{proposition}[Stationarity and mixing imply ergodicity]\label{MixingErg}
Let $0 \neq \Gamma \le \Z^N$ be a subgroup and let the probability space $\pspace$ be endowed with the stationary process $Z=\{Z(s): s\in \Gamma \}$ for which each $Z(s)$ takes values in $(S,\fS)$ and which fulfills the strong mixing condition from Equation \eqref{StrongSpatialMixing}. Then $Z$ is ergodic.
\end{proposition}
\begin{proof}
Let $A\in\cI$ be an $T$-invariant set of paths of $Z$, it suffices to show that $\p(A) \in \{0,1\}$, i.e.
\[
		\p_Z (A) = \p_Z( A \cap T_x A) \rightarrow \p_Z (A) \p_Z( T_x A) = \p_Z (A)^2 \text{ as } x\rightarrow \infty.
\]
Let $\epsilon >0$ be given and let $A,B \in \otimes_{k\in \Gamma} \fS$ be two sets of paths of $Z$. Then by Carathéodory's extension theorem there are $m,n\in \Z$ such that there are $A^m \in \otimes_{\substack{ k \in \Gamma, \\k \le m\cdot e_N } } \fS$ and $B^n \in \otimes_{ \substack{ k \in \Gamma, \\ k \ge n\cdot e_N} } \fS$ with the property that both
\[
			\p_Z (A \triangle A^m) < \frac{\epsilon}{5} \text{ and } \p_Z( B \triangle B^n) < \frac{\epsilon}{5}.
\]
Furthermore, by the strong mixing property from Equation \eqref{StrongSpatialMixing} there is an $x^* = r\cdot e_N \in \Z^N$ such that for $x \ge x^*$, $x\in \Gamma $ we have
\[
			| \p_Z (A^m \cap T_x B^n) - \p_Z (A^m) \p_Z(T_x B^n) |  < \frac{\epsilon}{5}.
\]
Consequently, we have for all $x \ge x^*$
\begin{align*}
 &\Big| \p ( Z\in A, Z \in T_x B ) - \p(Z \in A) \, \p(Z \in T_x B) \Big|	\\
&\le \p(Z\in A\setminus A^m, Z\in T_x B) + \p( Z\in A^m, Z\in T_x B\setminus B^n ) \\
&\quad + \Big| \p(Z\in A^m, Z \in T_x B^n ) - \p( Z \in A^m) \,\p( Z \in T_x B^n) \Big| \\
&\quad  + \p(Z \in A^m) \,\p( Z\in T_x B\setminus B^n) + \p(Z\in A\setminus A^m)\, \p(Z\in T_x B) < \epsilon.
\end{align*}
\end{proof}

The main result in this section is the following one which generalizes Birkhoff's one-dimensional ergodic theorem
\begin{theorem}\label{MixingImpliesErgodicity}
Let $0 \neq \Gamma \le \Z^N$ be a nontrivial subgroup and $\{Z(s): s\in \Gamma \}$ be a homogeneous strong mixing random field on $\pspace$ for some dimension $N\in\N_+$. Let $(n(k): k\in \N) \subseteq \N^N$ be an increasing sequence such that $e_N \le n(k) \le n(k+1)$ for which at least one coordinate converges to infinity. Then the sequence of index sets $I_{n(k)} := \{ z \in \Gamma: e_N \le z \le n(k) \}$ is admissible in the sense of Theorem \ref{ergodicTheorem}. In particular, we have
\begin{align*}
		\frac{1}{|I_{n(k)} |} \sum_{s \in I_{n(k)} } Z(s) \rightarrow \E{Z(e_N) } \quad a.s. \text{ as } k\rightarrow \infty.
\end{align*}
\end{theorem}
\begin{proof}
Since any subgroup of $\Z^N$ is isomorphic to $\Z^u$ for $0\le u \le N$, $u\in\N$, it suffices to consider the case $\Gamma = \Z^N$, $N\in\N_+$. In this case one computes easily that the regularity conditions of Theorem \ref{ergodicTheorem} are satisfied. The conclusion follows then from this theorem in combination with Proposition \ref{MixingErg}.
\end{proof}



\end{document}